\documentclass{amsart}
\usepackage{amssymb}
\usepackage{tikz}
\newcommand*{\DashedArrow}[1][]{\mathbin{\tikz [baseline=-0.25ex,-latex, dashed,#1] \draw [#1] (0pt,
0.5ex) -- (1.3em,0.5ex);}}

\usepackage{amscd}
\usepackage{empheq,amsmath}
\usepackage{graphicx}
\usepackage[cmtip,all]{xy}
\usepackage{multirow}
\usepackage{anysize} 

\marginsize{3cm}{3cm}{3cm}{3cm}

\newtheorem{theorem}{Theorem}[section]
\newtheorem{Lemma}[theorem]{Lemma}

\theoremstyle{definition}

\newtheorem{definition}[theorem]{Definition}

\newtheorem{prob}{Problem}
\newtheorem{claim}[theorem]{Claim}

\theoremstyle{proposition}
\newtheorem{proposition}[theorem]{Proposition}

\theoremstyle{remark}
\newtheorem{remark}[theorem]{\textbf{Remark}}

\numberwithin{equation}{section}

\theoremstyle{main}
\newtheorem{main}{Theorem}

\def\C{\mathbb C}

\begin{document}


\title[Branched pull-back components]{Branched pull-back components of the space of codimension $2$ foliations on $\mathbb P^4$}


\subjclass{37F75 (primary); 32G34, 32S65 (secondary)} 

\author[W. Costa e Silva]{W. Costa e Silva}
\address{IMPA, Est. D. Castorina, 110, 22460-320, Rio de Janeiro, RJ, Brazil}
\email{wancossil@gmail.com}

%
%
\date{}



\begin{abstract} We present a new list of irreducible components of the space of codimension two holomorphic foliations on  $\mathbb P^{4}$. They are associated to the pull-back by branched rational maps of 1-dimensional foliations on $\mathbb P^3$ leaving $2$-dimensional planes invariant.
\end{abstract}
\keywords{Lotka-Volterra, Pull-back foliation, Kupka singularity, Irreducible components}

\maketitle

%
\setcounter{tocdepth}{1}
\tableofcontents \sloppy

\section{Introduction}

It is a well known fact that on $\mathbb P^{n}$, the $n$-dimensional complex projective space, a codimension $q$ singular holomorphic foliation $\mathcal F$ is given by an element of $H^{0}(\mathbb P^{n},\Omega^{q}_{\mathbb P^{n}}\otimes \mathcal O_{\mathbb P^n}(\Theta
+q+1))$, where $\Theta
$ (called the degree of $\mathcal{F}$)  is  the degree of the tangency divisor between the foliation and a generic $\mathbb{P}^{q}$ linearly embedded in $\mathbb P^n$.  
Furthermore, an element of  $H^{0}(\mathbb P^{n},\Omega^{q}_{\mathbb P^{n}}\otimes \mathcal O_{\mathbb P^n}(\Theta
+q+1))$ can be represented, in homogeneous coordinates, by a polynomial  $q$-form $\eta$ on $\mathbb C^{n+1}$ with homogeneous coefficients of degree $\Theta+1$ such that $i_{R}\eta=0$, where $R=\sum_{j=0}^{n}x_j\frac{\partial}{\partial x_j}$ is the  radial vector field on $\mathbb C^{n+1}$.

The projectivisation of the set of homogeneous polynomial integrable $q$-forms $\eta$  as above, which have singular set of codimension greater than or equal to two defining in homogeneous coordinates codimension $q$ foliations on $\mathbb P^n$ will be denoted by $\mathbb{F}{\rm{ol}}\left(\Theta;n-q,n\right)$. Due to the integrability condition the set $\mathbb{F}{\rm{ol}}\left(\Theta;n-q,n\right)$ is a quasi-projective algebraic subset of $\mathbb PH^{0}(\mathbb P^{n},\Omega^{q}_{\mathbb P^{n}}\otimes \mathcal O_{\mathbb P^n}(\Theta+q+1))$.  A natural problem that arises is:
\vskip0.2cm
\begin{prob}\label{pr:0}
{\rm Describe the irreducible components of $\mathbb{F}{\rm{ol}}\left(\Theta;n-q,n\right)$  on ${\mathbb P^n}$, where $\Theta
\geq 0$, $n\geq3$ and $1\leq q\leq n-1$. }
\end{prob}

\vskip0.2cm What little that is known thus far about the irreducible components proceeds first by describing what they are. The classification of the irreducible components of $\mathbb{F}{\rm{ol}}\left(0;n-q,n\right)$ for all $1\leq q\leq n-1$ was given in \cite[Th. 3.8 p. 46]{cede} (a codimension $q$  foliation of degree zero on $\mathbb P^{n}$ is defined by a linear projection from $\mathbb P^n$ to $\mathbb P^{q}$). The classification of the irreducible components of $\mathbb{F}{\rm{ol}}\left(1;n-q,n\right)$ for all $1\leq q\leq n-1$, which require too many details to be explained here, was  given in \cite[Th. 6.2 and Cor. 6.3 p. 935-936]{lpt}. 

The space of codimension one foliations of degree $\Theta$ on $\mathbb P^n$ is denoted by $\mathbb{F}{\rm{ol}}\left(\Theta,n-1,n\right)$ and the study of these spaces was begun by \cite{jou}, where  the irreducible components of $\mathbb{F}{\rm{ol}}\left(\Theta,n-1,n\right)$ for $\Theta=0$ and $\Theta=1$ are described. Up to date there are partial lists of irreducible components of $\mathbb{F}{\rm{ol}}\left(\Theta,n-1,n\right)$, for instance, (\cite{omegar},\cite{ccgl},\cite{clne},\cite{costa},\cite{cupe}) which is known to be complete only when $\Theta\leq2$.

 In the  paper \cite{cln}, the authors proved that $\mathbb{F}{\rm{ol}}\left(2,n-1,n\right)$ has six irreducible components, which can be described by geometric and dynamic properties of a generic element. 

For the case $q\geq2$, the description of the irreducible components of $\mathbb{F}{\rm{ol}}\left(\Theta;n-q,n\right)$ we have a few known results: rational components \cite{cupeva}, foliations associated to affine Lie algebras, foliations induced by group actions, linear pull-backs (\cite{constant},\cite{cupe}) and generic pull-backs \cite{costaln}.

When we study the components of the space $\mathbb{F}{\rm{ol}}\left(\Theta;n-q,n\right), n\geq 3$ we discover that there are families of irreducible components in which the typical element is a pull-back of a foliation of degree $d$ on ${\mathbb P^{q+1}}$ by a nonlinear rational map. More precisely, the situation is as follows: Given a generic rational map $f: {\mathbb P^n}  \DashedArrow[->,densely dashed    ]   {\mathbb P^{q+1}}$, $1\leq q\leq n-2$ and $n\geq 3$, of degree $\nu\geq2$ and a $1$-dimensional foliation $\mathcal G$ of degree $d$ on $\mathbb P^{q+1}$, then it can be associated to the pair $(f,\mathcal G)$ the pull-back foliation $\mathcal F=f^{\ast}\mathcal G$. If $f$ and $\mathcal G$ are generic then the degree of the foliation $\mathcal F$ is $\Theta(\nu,d,q)=(d+q+1)\nu-q-1$.
Let $PB(\nu,d,q,n)$, be the closure in $\mathbb{F}{\rm{ol}}\left(\Theta;n-q,n\right)$, $\Theta=\Theta(\nu,d,q)$, of the set of this kind of foliations. The main result contained in \cite{costaln} is:

 \begin{theorem}\label{teo1.1}  \cite{costaln} The set $PB(\nu,d,q,n)$ is a unirational irreducible component of\break $\mathbb{F}{\rm{ol}}\left(\Theta;n-q,n\right)$ for all $\nu\geq2$, $d\geq2$, $1\leq q\leq n-2$ and $n\geq 3$.
\end{theorem}

This result is a generalization of the work \cite{clne} where the same problem has been considered in the case of codimension one foliations.   In \cite{costaln} is posed the following problem:

\begin{prob}\label{pr:x}
{\rm Is there a generalization of Theorem \ref{teo1.1}  in the case of pull-backs by branched maps (not generic) like in \cite{costa}?} 
\end{prob}

The maim aim of this work is to give a positive answer to problem \ref{pr:x} in the situation $n=4$ and $q=2$. The restriction to this situation is due to technical difficulties which appear along the proof.  Before we state our main result let us describe, briefly, the type of pull-back foliation that we shall consider. Let us fix some coordinates $z=(z_0,...,z_{4})$ on $\mathbb C^{5}$ and $x=(x_0,...,x_{3})$ on $\mathbb C^{4}$ and denote by $f: {\mathbb P^4}  \DashedArrow[->,densely dashed    ]   {\mathbb P^{3}}$ a rational map represented in the coordinates $z\in \mathbb{C}^{5}$ and $x \in \mathbb C^{4}$ by $\tilde f=(F_{0}^\alpha,F_{1}^\beta,F_{2}^\gamma,F_{3}^\delta)$ where $F_{i} \in \mathbb C[z]$ are irreducible homogeneous polynomials without common factors satisfying the relation $\alpha.\deg(F_0)=\beta.\deg(F_1)=\gamma.\deg(F_2)=\delta.\deg(F_3)=\nu\geq2$. In order for our techniques work in several steps of the proof we need to put arithmetical hypotheses on $\alpha$, $\beta$, $\gamma$ and $\delta$. They are divided into: 

\begin{itemize}
\item case (1): $1<\alpha<\beta<\gamma<\delta$ and $\alpha$, $\beta$, $\gamma$ and $\delta$ are pairwise relatively prime;

\item case (2): $1=\alpha<\beta<\gamma<\delta$ and $\beta$, $\gamma$ and $\delta$ are pairwise relatively prime;

\item case (3): $1=\alpha=\beta<\gamma<\delta$ and $\gamma$ and $\delta$ are relatively prime;

\item case (4): $1=\alpha=\beta=\gamma<\delta$. 
\end{itemize}

If $\mathcal G$ is a $1$-dimensional foliation of degree $d$ on $\mathbb P^3$ that leaves invariant the four coordinate planes $(x_0x_1x_2x_3=0)$, then $\mathcal G$ is given by a homogeneous polynomial $2$-form as below: 
\begin{multline*}
\Omega=x_{2}x_{3}P_{2,3}dx_{0}\wedge dx_{1}-x_{1}x_{3}P_{1,3}dx_{0}\wedge dx_{2}+x_{1}x_{2}P_{1,2}dx_{0}\wedge dx_{3}\\ +x_{0}x_{3}P_{0,3}dx_{1}\wedge  dx_{2}-x_{0}x_{2}P_{0,2}dx_{1}\wedge dx_{3}+x_{0}x_{1}P_{0,1}dx_{2}\wedge dx_{3}
\end{multline*}
where the $P_{i,j}$ are homogeneous polynomials in $\mathbb C[x_0,x_1,x_2,x_3]$ of degree $d-1$. A 1-dimensional singular holomorphic foliation of this type will be called a {\it{generalized Lotka-Volterra foliation}}. 
Let us denote by $\mathcal L\mathcal V(d,3)$ the space of this kind of foliations. Note that the intersections $(x_i=x_j=0)$ with $i\neq j$ are also invariant by any foliation $\mathcal G$ in $\mathcal L\mathcal V(d,3)$.
For a generic choice of $f$ and $\mathcal G$, the pull-back foliation $\mathcal F=f^{\ast}\mathcal G$ associated to the pair $(f,\mathcal G)$ has degree $\Theta_{\nu,d}^{\alpha,\beta,\gamma,\delta}=\nu\left[\left(d-1\right)+\frac{1}{\alpha}+\frac{1}{\beta}+\frac{1}{\gamma}+\frac{1}{\delta}\right] - 3$. 

Let us denote by $PB(\Theta_{\nu,d}^{\alpha,\beta,\gamma,\delta},2,4)$  the closure in $\mathbb{F}{\rm{ol}}\left(\Theta_{\nu,d}^{\alpha,\beta,\gamma,\delta},2,4\right)$ of the set of foliations $\mathcal F$ of the form $f^{\ast}\mathcal G$. 
If $\alpha, \beta$, $\gamma$ and $\delta$ are in the cases $(1)$, $(2)$, $(3)$  and $(4)$ respectively we are able to establish the following: 

\begin{main}\label{teob}$ PB(\Theta_{\nu,d}^{\alpha,\beta,\gamma,\delta},2,4)$ is a unirational irreducible component of  $\mathbb{F}{\rm{ol}}\left(\Theta_{\nu,d}^{\alpha,\beta,\gamma,\delta},2,4\right)$ for all $d\geq2$ and $\nu\geq2$. \end{main} 
Another result we would like to establish is related to {\it{generalized Lotka-Volterra foliations}} with generic properties. The first result in this direction was proved by Jouanolou in \cite{jou}, there he was able to prove that for all $d\geq 2$ there exists a generic subset of foliations on $\mathbb P^2$ having no algebraic invariant curves. After that, in \cite{ln} is proved that this set is open as well.  Regarding foliations leaving invariant lines in \cite{costa} and \cite{lnssc} were established similar results for foliations on $\mathbb P^2$ leaving invariant one and three lines respectively. These results were fundamental to construct the branched pull-back components for the space of codimension one foliations as described in \cite{costa}. In higher dimension, regarding generic properties of foliations with algebraic invariant sets of positive dimensions were known so far. In this direction, motivated by the works of \cite{costa} and \cite{lnssc} we will prove the following result that will be useful to establish Theorem \ref{teob}:  

\begin{main} \label{teoB}  Let $d\geq 2$. There exists a dense subset $\mathcal M\left(d,3\right)\subset \mathcal L\mathcal V(d,3)$, such that if $\mathcal G \in \mathcal M\left(d,3\right)$  then the only invariant algebraic sets of $\mathcal G$ are the 2-planes $(x_0x_1x_2x_3=0)$ and the intersections $(x_i=x_j=0)$ for $i\neq j$.
 \end{main}
Let us indicate some differences between the work \cite{costaln} and the current situation.  
In \cite{costaln}, the authors work with a class of generic $1$-dimensional foliations in $\mathbb P^{n}$ which does not have any algebraic invariant set of positive dimension. This is guaranteed by the works of \cite{cope} and \cite{lnsoares}. As a consequence, in that work, a generic pull-back foliation has no algebraic leaf.  On the other hand, in the current situation, the $1$-dimensional foliations that we have to consider are leaving $2$-dimensional planes invariant and consequently the generic elements of  $PB(\Theta_{\nu,d}^{\alpha,\beta,\gamma,\delta},2,4)$ will always have invariant algebraic surfaces. The quantity of such algebraic leaves will depend of cases $(1),\dots,(4)$. For instance, in case (1), a generic element of $PB(\Theta_{\nu,d}^{\alpha,\beta,\gamma,\delta},2,4)$ will have four algebraic leaves while in case (4) a generic element will have only one. The presence of branched maps makes the singular set of pull-back foliations more degenerate and therefore the study of the quasi-homogeneous singularities that arise in this problem leads us to use weighted projective spaces techniques and also weighted blow-ups which is precisely where the assumptions on $\alpha$, $\beta$, $\gamma$ and $\delta$ are essential. This exemplifies that our results were not known previously. 

\section{Generalized Lotka-Volterra foliations and the proof of Theorem \ref{teoB}}\label{Generalized Lotka-Volterra foliations}

Initially, let us observe that if $\mathcal G$ is a $1$-dimensional foliation of degree $d$ on $\mathbb P^3$, alternatively, it can be described in affine coordinates $(w_1,w_2,w_3)$ by a vector field of the form 
 $X=gR+\sum_{l=0}^{d}X_{l}$ where $g$ is a homogeneous polynomial of degree $d$, $R$ is the radial vector field $w_1\frac{\partial}{\partial w_1}+w_2\frac{\partial}{\partial w_2}+w_3\frac{\partial}{\partial w_3}$ and $X_l$ is a vector field whose components are homogeneous polynomials of degree $l$, $0\leq l\leq d$. If $g\not\equiv0$ then $X$ has a pole of order $d-1$ at infinity. If $\Gamma\subset \mathbb P^3$ is an irreducible algebraic invariant curve under $\mathcal G$, we say that  $\Gamma$ is an algebraic solution of $\mathcal G$ if  $\Gamma\backslash Sing(\mathcal G)$ is a leaf of the foliation where $Sing(\mathcal G)$ denotes the singular set of $\mathcal G$. In what follows, by invariant algebraic set of $\mathcal G$ we mean either an algebraic solution or an algebraic surface $\mathcal S\subset \mathbb P^3$ invariant by the foliation.  By a {\it{generalized Lotka-Volterra foliation}} of generic type and degree $d\geq2$ we mean a foliation represented by a vector field as above and such that:
 
 \begin{enumerate}
 \item at each $p\in Sing(X)$ we have $detDX(p)\neq0$, 
 \item  if $\{ \lambda_1(p), \lambda_2(p), \lambda_ 3(p)\}$ are eigenvalues of  $DX(p)$ then they satisfy $\frac{\lambda_i(p)}{\lambda_j(p)}\notin \mathbb R^{+}$(is not a positive real number for $j\neq i$),
\item a finite number of sums of ``residues'' (which are rational functions involving $\lambda_1(p), \lambda_2(p)$ and $\lambda_ 3(p)$), associated to the foliation at singular points, are not certain positive integers.
 \end{enumerate}

 These are sufficient conditions for a {\it{generalized Lotka-Volterra foliation}} to have no invariant algebraic sets different to the planes $(x_0x_1x_2x_3=0)$ and the intersections $(x_i=x_j=0)$ for $i\neq j$. First recall that if a smooth algebraic curve is invariant by a foliation on $\mathbb P^3$ the curve must contain a singular point of the foliation, for otherwise we get a holomorphic foliation with a compact leaf, which is impossible. Now suppose we have an algebraic invariant curve; then $(2)$ says that this curve can not have singular analytic nor smooth tangent branches at each of its singular points and also that the number of branches at a singular point is bounded by three, so we are reduced to consideration of invariant algebraic curves whose singularities, if any, have only smooth analytic no two of which are tangent. In this case we bring in a Theorem due to D. Lehmann \cite{leh}. The idea is that certain characteristic classes of bundles associated to the ambient complex manifold and to the foliation, as well as to invariant submanifolds, ``localize'' near the singular set of the foliation, giving rise to residues computable through local data for the foliation and whose sum give characteristic numbers of these bundles. 
 Condition $(3)$ means precisely that the sum of residues cannot be a characteristic number associated to a convenient bundle, thus ruling out the existence of certain algebraic solutions. We use the same idea to the case of existence of surfaces invariant by the foliation.

Let us denote by $\mathbb{F}{\rm{ol}}\left(d;1,3\right)$ the space of one dimensional foliations on $\mathbb P^{3}$ and denote by $ND(d;1,3)$ the subset of the foliations which have nondegenerate singularities. It is worth pointing out that for any $\mathcal G\in ND(d;1,3)$ it has exactly $N=\frac{d^{4}-1}{d-1}$ singularities. The proof of Theorem \ref{teoB} run as follows: First of all, from \cite{lnsoares} (p.670-p.672) we know that there exists an open, dense and connected subset $\mathcal W\subset ND(d;1,3)$ such that for any foliation $\mathcal G$ in $\mathcal W$ the linear part at each singularity has distinct eigenvalues and furthermore at each $p\in Sing(\mathcal G)$ we have $\frac{\lambda_i}{\lambda_j}\notin \mathbb R^{+}$. Let us denote by $\mathcal A(d;1,3)=\mathcal L\mathcal V(d,3)\cap\mathcal W$.

 First of all, we will recall the following lemma that can be found in \cite{lnsoares} p.$668$. 
 
 \begin{Lemma}\label{marcio}   $ND(d;1,3)$ is open, dense and connected in  $\mathbb {F}ol(1,d,3)$. Moreover, given $\mathcal G_0\in ND(d;1,3)$ with $sing(\mathcal G_0)=\{p_1,\dots,p_N\}$ there are neighborhoods $\mathcal U_0$ of $\mathcal G_0$  in $\mathbb {F}ol(1,d,3)$, $V_l$ of $p_l$ in  $\mathbb P^3$ and analytic functions $\psi_l:\mathcal U_0\to V_l$, $l=1,\dots, N$ such that $V_l\cap V_m\neq \emptyset$, $l\neq m$, and for any $\mathcal G \in \mathcal U_0$,  $\psi_l(\mathcal G)$ is the unique singularity of $\mathcal G$ in $V_l$. 
 \end{Lemma}
   \begin{proof}[Proof of Theorem B]
  For each $l\in\{1,\dots,N\}$ and $1\leq j\leq3$ we denote by $\lambda_j(\psi_l(\mathcal G,p_l))$ the $j^{th}$eingenvalue of the linear part of the singularity $\psi_l(\mathcal G,p_l)$ as in Lemma \ref{marcio}.
   Now we present a list of conditions relating the eigenvalue of the foliation $(\mathcal G,p_l)$ in the point $\psi_l(\mathcal G,p_l)$:
  \begin{enumerate}
   \item  The Kupka condition:
${\lambda_1(\psi_l(\mathcal G,p_l))}+{\lambda_2(\psi_l(\mathcal G,p_l))}+{\lambda_3(\psi_l(\mathcal G,p_l))}\neq0$.\\ 
 \noindent The conditions of being hyperbolic are as follows:\\
  \item $\frac{\lambda_1(\psi_l(\mathcal G,p_l))}{\lambda_2(\psi_l(\mathcal G,p_l))}\notin\mathbb R$  
  \item $\frac{\lambda_1(\psi_l(\mathcal G,p_l))}{\lambda_3(\psi_l(\mathcal G,p_l))}\notin\mathbb R$
  \item $\frac{\lambda_2(\psi_l(\mathcal G,p_l))}{\lambda_3(\psi_l(\mathcal G,p_l))}\notin\mathbb R$
       \noindent Let us now treat the condition on the residues. Define
  \begin{equation*}\label{eq1}
  \sigma_l^{i,j,k}(\mathcal G,p_l))=\frac{\lambda_j(\psi_l(\mathcal G,p_l))+\lambda_k(\psi_l(\mathcal G,p_l))}{\lambda_i(\psi_l(\mathcal G,p_l))}
  \end{equation*} 
  
    \begin{equation*}\label{eq2}
  \zeta_l^{i,j,k}(\mathcal G,p_l)=\frac{[\lambda_i(\psi_l(\mathcal G,p_l))]^{2}}{\lambda_j(\psi_l(\mathcal G,p_l))\lambda_k(\psi_l(\mathcal G,p_l))}
    \end{equation*}  where $(i,j,k)\in\{(1,2,3),(2,1,3),(3,1,2)\}$ and let
  $I_q=\{m\in\mathbb N: m\geq q\}$.
  
 We need that these numbers also satisfy the following conditions:

 \item  $\sigma_l^{1,2,3}(\mathcal G,p_l)\notin I_1$
 \item  $\sigma_l^{1,3,2}(\mathcal G,p)\notin I_1$
 \item  $\sigma_l^{2,3,1}(\mathcal G,p_l)\notin I_1$
 \item  $\zeta_l^{1,2,3}(\mathcal G,p_l)\notin I_1$

  \item  $\zeta_l^{1,3,2}(\mathcal G,p_l)\notin I_1$
 \item  $\zeta_l^{2,3,1}(\mathcal G,p_l)\notin I_1$

  \item  $[\sigma_l^{1,2,3}(\mathcal G,p_l)+\sigma_l^{2,1,3}(\mathcal G,p_l)+\sigma_l^{3,1,2}(\mathcal G,p_l)]\notin I_1$
 
  \item $[\sigma_l^{1,2,3}(\mathcal G,p_l)+\sigma_l^{1,3,2}(\mathcal G,p_l)]\notin I_1$
  \item $[\sigma_l^{1,2,3}(\mathcal G,p_l)+\sigma_l^{2,3,1}(\mathcal G,p_l)]\notin I_1$
  \item $[\sigma_l^{1,3,2}(\mathcal G,p_l)+\sigma_l^{2,3,1}(\mathcal G,p_l)]\notin I_1$
  
  \end{enumerate}
To simplify, let
 \begin{itemize}
\item  $H_i:=(x_i=0)$, $0\leq i\leq 3$ (the invariant $2$-planes)
 \item  $L_{ij}:=(x_i=x_j=0)$, $0\leq i<j\leq 3$ (the invariant lines)
 \item $p_{ijk}=(x_i=x_j=x_k=0)$, $0\leq i<j<k\leq 3$ 
 \end{itemize}

To continue let us denote four subsets of foliations as follows:
 \begin{enumerate}
 \item $\mathcal W_0^{\ast}=\{ \mathcal G \in \mathcal A\left(d,3\right)$ which satisfies the conditions $2,3$ and $4$ at the singularity $p_{ijk}\in H_i\cap H_j\cap H_k\}$.
 \item $\mathcal W_1^{\ast}=\{ \mathcal G \in \mathcal A\left(d,3\right)$ which satisfies the conditions $2,3$ and $4$ at the singularity $p_{ij}\in (L_{ij})\backslash (H_i\cap H_j\cap H_k)\}$.
 \item $\mathcal W_2^{\ast}=\{ \mathcal G \in \mathcal A\left(d,3\right)$ which satisfies the conditions $2,3$ and $4$ at the singularity $p_{i} \in H_i\backslash(L_{ij}\cup L_{ik}\cup L_{jk})\}$.
 \item $\mathcal W_3^{\ast}=\{ \mathcal G \in \mathcal A\left(d,3\right)$ which satisfies the conditions $2,3$ and $4$ at the singularity $p_{m}\notin H_i\cup H_j \cup H_k\}$.
 \end{enumerate}
  By a construction analogous to the one found in (\cite[section 3, p.669-670]{lnsoares}) we have that each $\mathcal W_i^\ast$, $0\leq i\leq 3$, is open and dense in $\mathcal A\left(d,3\right)$. Furthermore, observe that any element in $\mathcal L\mathcal V(d,3)$ has symmetries. Hence, we have that any foliation in $\mathcal A\left(d,3\right)$  belongs to one of those $\mathcal W_i^\ast$, $0\leq i\leq 3$.
  Therefore for our purpose it suffices to study a element $\mathcal G$ on an affine chart.  \par
Let us now show that these sets are non-empty. For this let us consider the Lotka-Volterra foliation of degree $2$ in an affine chart $(\mathbb C^3, w)$ where $w=(w_1,w_2,w_3)$:

\begin{empheq}[left=\empheqlbrace]{align}\label{afim}
 \dot{w_1}&= w_1(w_1+ (-i-\sqrt{2})) \nonumber \\\nonumber 
 \dot{w_2}&= w_2(iw_2+ 4i) \nonumber \\\nonumber 
 \dot{w_3}&= w_3(w_3+1)
\end{empheq}\\
\\

It is enough to choose 4 singularities satisfying conditions $(1)-(14)$.

We do necessary computations using Mathematica; the results are shown in Table $1$ as the reader can see. For this equation the singular points described in the table are: $p_{123}=(0,0,0)  ,p_{12}=(0,-4,0) ,p_{1}=(0,-4,-1) $   and $p_{m}=(i+\sqrt{2},-4,-1)$.

For each $0\leq i\leq3$,  we define $\mathcal W_i$ to be the subset of foliations $\mathcal G \in \mathcal W_i^\ast$ that also satisfy the conditions $\{1,5,\dots,14\}$. Of course, each $\mathcal W_i^{\ast}\cap N(1;d,3)$ is dense in $\mathcal W_i$ because negating the conditions $\{1,5,\dots,14\}$ gives a denumerable set of analytic subsets of codimension 1 of $\mathcal W_i^{\ast}$. We denote $\bigcap_{i=0}^{i=3}\mathcal W_i$ by $\mathcal M(d,3)$. The previous example shows that each one of those $\mathcal W_i$ is non-empty.

{\footnotesize
\begin{table}[]
\centering

\label{tabela}
\begin{tabular}{|r | c | c | c | c |}
\hline
Singularity  & $p_{123}\in H_1\cap H_2\cap H_3$ & $p_{12}\in (L_{11})\backslash (H_i\cap H_2\cap H_3)$ & $p_{1}\in H_1\backslash(L_{12}\cup L_{13}\cup L_{13})$ & $ p_{m}\notin H_i\cup H_j \cup H_k$\\
\hline
\cline{1-5}
\cline{2-5}
                       &$\lambda_{1}=4i $ \ \ \ \ \ \ \ \ \               &$\lambda_{1}=-4i$    \ \ \ \ \ \       &$\lambda_{1}=-4i$  \ \ \ \ \ \ \ \ \  &$\lambda_{1}=-4i$ \ \ \ \ \ \ \ \ \  \\
\cline{2-5}
  Eigenvalues &$\lambda_{2}=-i-\sqrt{2}$ &$\lambda_{2}=-i-\sqrt{2}$&$\lambda_{2}=-i-\sqrt{2}$  &$\lambda_{2}=i+\sqrt{2}$\\ 
  \cline{2-5}
                       &$\lambda_{3}=1$  \ \ \ \ \ \ \ \ \ \          &$\lambda_{3}=1$    \ \ \ \ \ \ \ \ \ \     &$\lambda_{3}=-1$   \ \ \ \ \ \  \ \ \ \  &$\lambda_{3}=-1$\ \ \ \ \ \ \ \ \  \\
\hline
Condition 1  &      $(1+3i)-\sqrt{2}$  & $ (1-5i)-\sqrt{2} $  & $(-1-5i)-\sqrt{2}$ & $(-1-3i)+\sqrt{2}$ \\
\hline
Condition 2  &        $-\frac{4}{3}i(-i+\sqrt{2})$    & $ \frac{4}{3}(1+i\sqrt{2})$ &  $ \frac{4}{3}(1+i\sqrt{2})$ & $-\frac{4}{3}i(-i+\sqrt{2})$ \\
\hline
Condition 3  &      $4i$ &  $-4i$ & $4i$ & $4i$ \\
\hline
Condition 4  &    $-i-\sqrt{2}$ & $-i-\sqrt{2}$ & $i+\sqrt{2}$ & $-i-\sqrt{2}$ \\
\hline
Condition 5  &     $3i-\sqrt{2}$  & $-5i-\sqrt{2}$  & $5i+\sqrt{2}$ & $3i-\sqrt{2}$ \\
\hline
Condition 6  &     $-\frac{1+4i}{i+\sqrt{2}}$ &  $-\frac{1-4i}{i+\sqrt{2}}$ & $\frac{1+4i}{i+\sqrt{2}}$ &$ -\frac{1+4i}{i+\sqrt{2}}$ \\
\hline
Condition 7  &      $\frac{i}{4}((-1+i)+\sqrt{2})$ & $\frac{1}{4}((1+i)-i\sqrt{2})$  & $-\frac{i}{4}((1+i)+\sqrt{2})$ & $\frac{i}{4}((-1+i)+\sqrt{2})$  \\
\hline
Condition 8  &      $\frac{16}{i+\sqrt{2}}$  &  $\frac{16}{i+\sqrt{2}}$ & $-\frac{16}{i+\sqrt{2}}$ & $\frac{16}{i+\sqrt{2}}$ \\
\hline
Condition 9  &      $-\frac{i}{4}+\frac{1}{\sqrt{2}}$ &  $\frac{i}{4}-\frac{1}{\sqrt{2}}$ &  $-\frac{i}{4}+\frac{1}{\sqrt{2}}$ & $-\frac{i}{4}+\frac{1}{\sqrt{2}}$  \\
\hline
Condition 10 &      $\frac{1}{4(1-i\sqrt{2})}$ & $ \frac{1}{4(-1+i\sqrt{2})} $  & $\frac{1}{4(-1+i\sqrt{2})}$ & $\frac{1}{4(i-\sqrt{2})}$  \\
\hline
Condition 11 &    $ -\frac{(23+15i)+(2-7i)\sqrt{2}}{4(i+\sqrt{2})}  $     & $ \frac{(7+15i)+(2-23i)\sqrt{2}}{4(i+\sqrt{2})}  $  & $ \frac{(-7+15i)+(2+23i)\sqrt{2}}{4(i+\sqrt{2})}  $ & $ -\frac{(23+15i)+(2-7i)\sqrt{2}}{4(i+\sqrt{2})}  $  \\
\hline
Condition 12 &     $\frac{2i((-2+3i)+\sqrt{2})}{i+\sqrt{2}}$ & $ 5i+\sqrt{2}+\frac{1+4i}{i+\sqrt{2}}$  & $\frac{2i((2+i)+3\sqrt{2})}{i+\sqrt{2}}$ & $\frac{2i((-2+3i)+\sqrt{2})}{i+\sqrt{2}}$ \\
\hline
Condition 13 &     $\frac{i}{4}((11+i)+(1+4i)\sqrt{2})$ & $\frac{1}{4}((1+19i)+(4-i)\sqrt{2})$ & $\frac{1}{4}((1+19i)+(4-i)\sqrt{2})$ & $\frac{i}{4}((11+i)+(1+4i)\sqrt{2})$ \\
\hline
Condition 14 &     $-\frac{(3+15i)+(2+i)\sqrt{2}}{4(i+\sqrt{2})}$ & $\frac{(5+15i)+(2-i)\sqrt{2}}{4(i+\sqrt{2})}$  & $\frac{(5+15i)+(2-i)\sqrt{2}}{4(i+\sqrt{2})}$  & $-\frac{(3+15i)+(2+i)\sqrt{2}}{4(i+\sqrt{2})}$ \\
\hline
\end{tabular}
\caption{Singularities and their conditions.}
\end{table}
}

 Therefore, $\mathcal M(d,3)$ is dense in $\mathcal A(d,3)$.
 
 To finish, it remains to prove that any $\mathcal G\in \mathcal M(d,3)$ does not have any algebraic invariant object of positive dimension other than the $2$-planes $H_i$ and the lines $L_{ij}$.

 \begin{enumerate}
 \item  Non-existence of invariant algebraic curves other than the lines $L_{ij}$.\\

 Suppose $\Gamma\subset \mathbb P^{3}$ is an irreducible curve whose singularities, in case they exist, are such that $\Gamma$ has only smooth analytic branches through  each of them. Assume $\Gamma$ is invariant by $\mathcal G$. Then $sing(\mathcal G)\cap \Gamma\neq \emptyset$ and moreover, if $p_l\in sing(\mathcal G)$ then the branches of $\Gamma$ through $p$ are transverse to each other. According to \cite[Lemma 2.4, p.147 and Remark 2.5 p.148]{soares}, we can associate to each branch an index (a type of residue) those  $\sigma_l^{i,j,k}(\mathcal G,p_l)$ and the contribution of all these index is an $\mathbb N$-linear combination of the $\sigma_l^{i,j,k}(\mathcal G,p_l)$, which is always a natural number. But in the way that we have chosen our singularities, our residues are not positive real numbers. Therefore it does not exist any invariant algebraic curve other than the lines $L_{ij}$. 
  \item Non-existence of invariant algebraic surfaces other than the $H_i$. \\
     Suppose $\mathcal S$ is an algebraic surface invariant by $\mathcal G$ other than the coordinate $2-$planes. On one hand, according to Lemma $5.2$ from \cite{soares} we have that or $S$ is smooth or it would have a curve of singular points for the foliation. Since we are considering foliations only with isolated singular points we conclude that $S$ has to be smooth. Now using the vanishing theorem \cite{leh} we must have that $\mathcal S\cap sing(\mathcal G)\neq \emptyset$. Hence  we can associate to each germ of a surface passing through a singular point $p_l$ an index (a type of residue), in fact an $\mathbb N$-linear combination of the $\zeta_l^{i,j,k}(\mathcal G,p_l)$ and the contribution of all these index has to be the projective degree of the surface which is a natural number according to \cite[Theorem 2.1 and Remark 2.2 p.145]{soares}. On the other hand, in the way the we chose our singularities, our residues are not positive real numbers. Hence there is no such $\mathcal S$ invariant under the foliation.
  
 \end{enumerate}
 To finish we just take the foliation $\mathcal G \in \mathcal M(2,3)$ as before and take the ramification $T:\mathbb P^3\to \mathbb P^3$ given by $T([x_0:x_1:x_2:x_3])=([x_0^{d-1}:x_1^{d-1}:x_2^{d-1}:x_3^{d-1}])$. The pull-back foliation $T^{\ast}\mathcal G$ is an element of $\mathcal M(d,3)$ because $T$ neither produces nor contracts any curve or surface. This finishes the proof.

 \end{proof}

\section {Rational maps}
Let $f : {\mathbb P^4}  \DashedArrow[->,densely dashed    ]   {\mathbb P^{3}}$ be a rational map, and let $\tilde{f}: {\mathbb C^{5}} \to {\mathbb C^{4}}$ its natural lifting in homogeneous coordinates. We characterize the set of rational maps used throughout this text as follows:

{\begin{definition} We denote by $RM^{\alpha,\beta,\gamma,\delta}\left(4,3,\nu\right)$ the set of the branched rational maps $f: \mathbb P^4  \DashedArrow[->,densely dashed    ]   \mathbb P^{3}$ of degree $\nu\geq 2$ given by $f=\left(F_{0}^\alpha:F_{1}^\beta:F_{2}^\gamma:F_{3}^\delta\right)$  where, the $F_{js}$, are irreducible homogeneous polynomials without common factors, satisfying $$\alpha.\deg F_{0}=\beta.\deg F_{1}=\gamma.\deg F_{2}=\delta.\deg F_3=\nu\geq2.$$ 
\end{definition}}

The \emph{indeterminacy locus} of $f$ is, by definition, the set \break$I\left(f\right)=\Pi_{4}\left(\tilde{f}^{-1}\left(0\right)\right)$, where $\Pi_{4}:\mathbb C^{5}\backslash \{0\}\to\mathbb P^{4}$ is the canonical projection. Observe that the restriction $f|_{\mathbb P^4 \backslash I\left(f\right)}$ is holomorphic. 

\begin{definition}\label{generic} We say that $f  \in RM^{\alpha,\beta,\gamma,\delta}\left(4,3,\nu\right)$ is $generic$ if for all  \break$p \in$ $\tilde{f}^{-1}\left(0\right)\backslash\left\{0\right\}$ we have $dF_{0}\left(p\right)\wedge dF_{1}\left(p\right)\wedge dF_{2}\left(p\right)\wedge dF_{3}\left(p\right) \neq 0.$  
\end{definition}

This is equivalent to saying that $f  \in RM^{\alpha,\beta,\gamma,\delta}\left(4,3,\nu\right)$ is $generic$ if $I(f)$ is the nontangential intersection of the four hypersurfaces $\Pi_4(F_{i}=0)$ for $i=0,...,3.$  Moreover if $f$ is generic and $\deg(f)=\nu$, then by Bezout's theorem $I\left(f\right)$ is discrete and  consists of $\frac{\nu^{4}}{\alpha\beta\gamma\delta}$ distinct points. The set of  generic branched rational maps of degree $\nu$ will be denoted by $Gen^{\alpha,\beta,\gamma,\delta}\left(4,3,\nu\right)$. The next result is standard in algebraic geometry.  \begin{proposition} $Gen^{\alpha,\beta,\gamma,\delta}\left(4,3,\nu\right)$ is a Zariski open and dense subset of $RM^{\alpha,\beta,\gamma,\delta}\left(4,3,\nu\right)$.
\end{proposition}

\section{Generic pairs and the description of the generic pull-back foliations}\label{par generico}

 \begin{definition} Let $f$ be an element of $Gen^{\alpha,\beta,\gamma,\delta}\left(4,3,\nu\right)$ and $\mathcal G$ $\in \mathcal M(d,3).$ We say that $(f,\mathcal G)$ is a generic pair if $[Sing(\mathcal G)\cap D(f)]=\emptyset$, where $$D(f):=[\tilde f\{z\in\mathbb C^5| dF_0(z)\wedge dF_1(z)\wedge dF_2(z)\wedge dF_3(z)\wedge dF_4(z)=0\}].$$
\end{definition}

As we know, the foliation $f^\ast \mathcal G$ is represented in homogeneous coordinates by $\tilde f^\ast \Omega$ where $\tilde f$ is the lifting of $f$ and $\Omega$ is the 2-form that represents the foliation $\mathcal G$ in homogeneous coordinates. Therefore $\tilde f^\ast \Omega$  has the following expression \begin{multline}\label{expressao homogenea}
\eta_{[f,\mathcal G]}=[ \alpha\beta F_{2}F_{3}(P_{23}\circ\tilde f) dF_{0}\wedge dF_{1}-\\\alpha\gamma F_{1}F_{3}(P_{13}\circ\tilde f)dF_{0}\wedge dF_{2}+\alpha\delta F_{1}F_{2}(P_{12}\circ\tilde f)dF_{0}\wedge dF_{3} +\\ \beta\gamma F_{0}F_{3}(P_{03}\circ\tilde f)dF_{1}\wedge  dF_{2}-  \beta\delta F_{0}F_{2}(P_{02}\circ\tilde f)dF_{1}\wedge dF_{3}+\\ \gamma\delta F_{0}F_{1}(P_{01}\circ\tilde f)dF_{2}\wedge dF_{3}]
\end{multline}

Since $P_{{i,j}}$ are homogeneous polynomials of degree $(d-1)$ and the $F_{{i}}$ satisfy the condition $\alpha.\deg F_{0}=\beta.\deg F_{1}=\gamma.\deg F_{2}=\delta.\deg F_3=\nu\geq2$ the coefficients of $\eta_{[f,\mathcal G]}$ are homogeneous of degree   $\nu[(d-1)+\frac{1}{\alpha}+\frac{1}{\beta}+\frac{1}{\gamma}+\frac{1}{\delta}]-2$. From these considerations we have
\begin{proposition}
 If $\mathcal F=f^\ast \mathcal G$ where $(f,\mathcal G)$ is a generic pair, then the degree of $\mathcal F$ is  $$\Theta_{\nu,d}^{\alpha,\beta,\gamma,\delta}=\nu[(d-1)+\frac{1}{\alpha}+\frac{1}{\beta}+\frac{1}{\gamma}+\frac{1}{\delta}]-3.$$
\end{proposition}

Set $\mathcal W=\{ \mathcal F, \mathcal F=f^\ast \mathcal G$ where $(f,\mathcal G)$ is a generic pair \} of generic pull-back foliations. We remark that it is a Zariski real open and dense subset of  $PB(\Theta_{\nu,d}^{\alpha,\beta,\gamma,\delta},2,4)$.

In the sequence we will describe the singular set of a generic pull-back foliation.
\subsection{Quasi-homogeneous singular set of $f^\ast \mathcal G$}\label{section5.2}
 Let us now describe  $\mathcal{F} = f^*(\mathcal{G})$ in a neighborhood of a point $p \in I(f).$

It is easy to show that there exists a local chart $(U,x)\subset (\mathbb C^4,0)$, $x=(x_0,...,x_{3})$  around $\tilde p\in \Pi_{4}^{-1}(p)$ such that the lifting $\tilde f$ of $f$ is of the form $\tilde f|_{U}=(x_0^\alpha,x_1^{\beta},x_2^{\gamma},x_{3}^{\delta}):U \to \mathbb{C}^{4}$. In particular the lifting $\mathcal {F}^\ast|_{U(p)}$ is represented by the quasi-homogeneous $2$-form 

\begin{multline}\label{eta}
\tilde\eta(x_0,...,x_{3})=\alpha\beta x_{2}x_{3}P_{23}(x_0^\alpha,x_1^\beta,x_2^\gamma,x_{3}^\delta)dx_{0}\wedge dx_{1}-\\\alpha\gamma x_{1}x_{3}P_{13}(x_0^\alpha,x_1^\beta,x_2^\gamma,x_{3}^\delta)dx_{0}\wedge dx_{2}+\alpha\delta x_{1}x_{2}P_{12}(x_0^\alpha,x_1^\beta,x_2^\gamma,x_{3}^\delta)dx_{0}\wedge dx_{3} +\\ \beta\gamma x_{0}x_{3}P_{03}(x_0^\alpha,x_1^\beta,x_2^\gamma,x_{3}^\delta)dx_{1}\wedge  dx_{2}-  \beta\delta x_{0}x_{2}P_{02}(x_0^\alpha,x_1^\beta,x_2^\gamma,x_{3}^\delta)dx_{1}\wedge dx_{3}+\\ \gamma\delta x_{0}x_{1}P_{01}(x_0^\alpha,x_1^\beta,x_2^\gamma,x_{3}^\delta)dx_{2}\wedge dx_{3}
\end{multline}
which is the pull-back via $\tilde f|_{U}$ of $\Omega$, the $2-$form defining the $1-$dimensional foliation on $\mathbb P^3$. According to \cite{ln3,ccgl} a $2$-form of this type will be called a $2$-dimensional quasi-homogeneous singularity of type $$(\beta\gamma\delta,\alpha\gamma\delta,\alpha\beta\delta,\alpha\beta\gamma;\alpha\beta\gamma\delta(d-1)).$$ In particular, note that $\eta$ is a quasi-homogeneous $2$-form invariant under the $\mathbb C^{*}$-action  
\begin{equation}
(x_0,x_1,x_2,x_3)\to(s^{\beta\gamma\delta}x_0,s^{\alpha\gamma\delta}x_1,s^{\alpha\beta\delta}x_2,s^{\alpha\beta\gamma}x_3).
\end{equation}

The point $0$ in $\mathbb C^4$ corresponding to $\eta=0$ will be denoted by $\mathcal C(\eta)$. It will be called the central point of quasi-homogeineity of  $\eta$. The union of all central points of quasi-homogeneity of the foliation $\mathcal F$ will be denoted by $\mathcal C(\mathcal F)$. These singularities were studied in \cite{ln3} and  according to \cite[Theorem 3, p. 653]{ln3}, $\mathcal C(\mathcal F)$ is stable under holomorphic perturbations in the following sense: if $\mathcal F_t$ is a deformation of $\mathcal F_0$ the set $\mathcal C(\mathcal F_t)$ of the central points of $\mathcal F_ t$ is a deformation of $\mathcal C(\mathcal F_0)$.  Since $\mathcal C(\mathcal F_0)=\{p_{1},\dots,p_{j},\dots,p_{\frac{\nu^4}{\alpha\beta\gamma\delta}}\}$ we may denote the deformation of $\mathcal C(\mathcal F_0)$ by $\mathcal C(\mathcal F_t)=\{p_{1}(t),\dots,p_{j}(t),\dots,p_{\frac{\nu^4}{\alpha\beta\gamma\delta}}(t)\}$.

\subsection{The Kupka set of $\mathcal{F} = f^*\mathcal{G}$}\label{section5.1}

 Now let $\mathcal{F} = f^*\mathcal{G}$ where $(f,\mathcal G)$ is a generic pair. Consider a point $\tau_i\in sing(\mathcal{G})$ and let us describe $\mathcal F$ in a neighborhood of $V_{\tau_i}=\overline{f^{-1}(\tau_i)}$. Observe first that  $V_{\tau_i}$ is an algebraic curve in $\mathbb P^{4}$ that contains $I(f)=\mathcal C(\mathcal F)$.  Such a curve is in fact  a complete intersection. To see this, suppose for instance, that $\Pi_3^{-1}(\tau_i)$ is the line $\ell=\{(T.x_0,T.x_1,T.x_2,T.x_3);T\in \mathbb C^{\ast}$, where $x_3\neq0$. Then $V_{\tau_i}$ is the complete intersection $V_{\tau_i}=\Pi_4(\{x_{3}F_{0}^{\alpha}-x_{0}F_{3}^{\delta}=x_{3}F_{1}^{\beta}-x_{1}F_{3}^{\delta}=x_{3}F_{2}^{\gamma}-x_{2}F_{3}^{\delta}=0\})$. Since $F$ is generic, it follows that the hypersurfaces   $\Pi_4(\{x_{3}F_{0}^{\alpha}-x_{0}F_{3}^{\delta}=0\})$, $\Pi_4(x_{3}F_{1}^{\beta}-x_{1}F_{3}^{\delta}=0\})$ and $\Pi_4(x_{3}F_{2}^{\gamma}-x_{2}F_{3}^{\delta}=0\})$ intersects transversely along $V_{\tau_i}$. 
 Let us take, for instance, the singularity $\tau_1=[0:0:0:1]$.  Now, fix $p\in V_{\tau_1}\backslash I(f)$. There exist local analytic coordinate systems $(U,(x_0,x_1,x_2,x_3)), U\subset\mathbb C^4$, and $(V,(y_1,y_2,y_3)), V\subset\mathbb C^{3}$, at $p$ and $\tau_1=f(p)$ respectively, such that  $f(x_0,x_1,x_2,x_3)=(x_0^{\alpha},x_1^{\beta},x_2^{\gamma})$, $y(\tau_1)=0$. Suppose that $\mathcal {G}$ is represented by the vector field 
 $Y=y_1(\lambda_1+h.o.t)\frac{\partial}{\partial y_1}+y_2(\lambda_2+h.o.t)\frac{\partial}{\partial y_2}+y_3(\lambda_3+h.o.t)\frac{\partial}{\partial y_3}$ in a neighborhood of $\tau_1$. Then $\mathcal F$ is represented by $f^\ast Y=x_0(\frac{\lambda_1}{\alpha}+h.o.t)\frac{\partial}{\partial x_0}+x_1(\frac{\lambda_2}{\beta}+h.o.t)\frac{\partial}{\partial x_1}+x_2(\frac{\gamma}{\lambda_3}+h.o.t)\frac{\partial}{\partial x_2}$. Note that $tr(D(f^\ast Y)(0))=\beta\gamma\lambda_1+\alpha\gamma\lambda_2+\alpha\beta\lambda_3\neq0$. It follows that in $U$, the foliation $\mathcal F$ is  biholomorphically equivalent to the product of two foliations of dimension one: the singular foliation induced by the vector field $f^\ast Y$ in $(\mathbb C^{3},0)$ and a regular foliation of dimension one. Therefore if $p$ is as before it belongs to the Kupka-set of $\mathcal {F}$ because $\mathcal G\in \mathcal M(d,3)$ and $Div(f^\ast Y(p))\neq0$. Note that this local product structure is stable under holomorphic deformations of $\mathcal F$ \cite[Th. A' p. 396]{medeiros}. For the other singularities the argument is analogous. 

\par Since $\mathcal G$ has degree $d$ and all of its singularities are non degenerate it has $N=d^3+d^2+d+1$ singularities, say, $\tau_1,...,\tau_N$. We will denote the curves $\overline{f^{-1}(\tau_1)},\dots,\overline{f^{-1}(\tau_N)}$ by $V_{\tau_1}, \dots,V_{\tau_N}$ respectively. For the other curves the argument is analogous. 
 We can summarise this discussion in the following:

\begin{proposition} For each $\{j=1,\dots,N\}$,  $V_{\tau_j}$ is a complete intersection of $3$ nontangential algebraic hypersurfaces. Furthermore,  $V_{\tau_j}\backslash I(f)$ is contained in the Kupka set of $\mathcal F=f^\ast\mathcal G$.
\end{proposition}   

\subsubsection{Deformations of the Kupka set of ${\mathcal F}_0= f_{0}^\ast \mathcal G_{0}$, where $(f_0,\mathcal G_ 0)$ is a generic pair.} 

We will state a Lemma which say that for any pull-back generic foliation $\mathcal{F}_{0}$ and any germ of deformations of foliations $(\mathcal{F}_{t})_{t \in (\mathbb{C},0)}$ such that $\mathcal{F}_{0}=\mathcal{F}_{t=0}$, $(\mathcal{F}_{t})_{t \in (\mathbb{C},0)}$ has a Kupka set with similar properties to that of $\mathcal{F}_{0}$ for all $t \in (\mathbb{C},0)$.  
  
\begin{Lemma} \label{subvariedades}
There exist $\epsilon >0$ and $C^{\infty}$ isotopies $\phi_{\tau_i}:D_{\epsilon}\times V_{\tau_i} \to \mathbb P^4,  \tau_i \in Sing(\mathcal G_{0})$, such that $V_{\tau_i}(t)=\phi_{\tau_i}(\{t\}\times V_{\tau_i})$ satisfies:
\begin{enumerate}

\item[(a)] $V_{\tau_i}(t)$ is an algebraic curve in $\mathbb P^4$ and  $V_{\tau_i}(0) = V_{\tau_i}$ for all $ \tau_i\in Sing(\mathcal G_{0})$ and for all $t \in D_{\epsilon}.$

\item[(b)]   $V_{\tau_i}(t) \backslash \mathcal C(\mathcal F_t)$ is contained in the Kupka-set of $\mathcal{F}_t$ for all $\tau_i \in Sing(\mathcal G_{0})$ and for all $t \in D_{\epsilon}.$ In particular, for fixed $t$, the transversal type of $\mathcal F_{t}$ is constant along $V_{\tau_i}(t) \backslash \mathcal C(\mathcal F_t)$.

\item[(c)] $\mathcal C(\mathcal F_t) \subset V_{\tau_i}(t)$ for all $\tau_i \in Sing(\mathcal G_{0})$ and for all $t \in D_{\epsilon} $. Moreover, if $\tau_i \neq \tau_j$, and $\tau_i, \tau_j \in Sing(\mathcal G_{0})$,  we have $V_{\tau_i}(t)  \cap V_{\tau_j}(t) = \mathcal C(\mathcal F_t)$ for all $t \in D_{\epsilon}$ and the intersection is nontangential.

\end{enumerate}
\end{Lemma}
\begin{proof} The argument is similar to \cite[Lemma 2.3.3, p.83]{ln} and uses essentially the local stability under deformations of the Kupka set of $\mathcal F_0$ and also of $\mathcal C(\mathcal F_0)$. \end{proof}

\section{Proof of Theorem \ref{teob}}
\subsection{Plan of the proof} \label{plano da prova}
To start with, $PB(\Theta_{\nu,d}^{\alpha,\beta,\gamma,\delta},2,4)$ is a unirational irreducible algebraic subset of $\mathbb{F}{\rm{ol}}\left(\Theta_{\nu,d}^{\alpha,\beta,\gamma,\delta},2,4\right)$, because it is the closure in $\mathbb{F}{\rm{ol}}\left(\Theta_{\nu,d}^{\alpha,\beta,\gamma,\delta},2,4\right)$ of the set \break $\{f^{\ast}\mathcal G|f\in RM^{\alpha,\beta,\gamma,\delta}\left(4,3,\nu\right), \mathcal G \in \mathcal L\mathcal V(d,3)\}$. Let $Z$ be the (unique)
 irreducible component of $\mathbb{F}{\rm{ol}}\left(\Theta_{\nu,d}^{\alpha,\beta,\gamma,\delta},2,4\right)$ containing $PB(\Theta_{\nu,d}^{\alpha,\beta,\gamma,\delta},2,4)$. Since $PB(\Theta_{\nu,d}^{\alpha,\beta,\gamma,\delta},2,4)$ and $Z$ are irreducible it is sufficient to prove that there exists $\mathcal F=f^\ast\mathcal G\in PB(\Theta_{\nu,d}^{\alpha,\beta,\gamma,\delta},2,4)$ such that for any germ of a holomorphic one parameter family $(\mathcal F_t)_{t \in D_{\epsilon}}$ of foliations $(\mathcal F_t)_{t \in D_{\epsilon}} \in Z$ with $\mathcal F_0=\mathcal F$,  $\mathcal F_t \in PB(\Theta_{\nu,d}^{\alpha,\beta,\gamma,\delta},2,4)$, for any $t\in D_{\epsilon}$.
 
  We choose $\mathcal F=f^\ast\mathcal G$, where $(f,\mathcal G)$ is a generic pair (see \S \ref{par generico}), and $\mathcal G\in \mathcal M(d,3)$ see (\S \ref{Generalized Lotka-Volterra foliations}). 
  
  Given the one parameter family $(\mathcal F_t)_{t \in D_{\epsilon}}$  with $\mathcal F_0=f_0^\ast\mathcal G_0$ we will construct in \S\ref{familia dos mapas} a one parameter family of generic maps, $(f_t)_{t \in D_{\epsilon}}$,  and in \S \ref{familia das folheacoes}  $(\mathcal G_t)_{t \in D_{\epsilon}}$  a family of foliations, such that $\mathcal F_t=f_t^\ast\mathcal G_t \in PB(\Theta_{\nu,d}^{\alpha,\beta,\gamma,\delta},2,4)$ for all ${t \in D_{\epsilon}}$. A problem with the families $(f_t)_{t \in D_{\epsilon}}$ and $(\mathcal G_t)_{t \in D_{\epsilon}}$ that we will construct is that we cannot assert a priori that $\mathcal F_t=f_t^\ast\mathcal G_t$ for any $t\in D_{\epsilon}$. This fact will be proved in \S \ref{ss:33}.
\subsection{Construction of the families $(f_{t})_{t \in D_{\epsilon}}$}\label{familia dos mapas} We will construct a family  of rational maps $f_{t}:{\mathbb{P}^4 \DashedArrow[->,densely dashed    ]   \mathbb{P}^{3}}$, $f_{t} \in Gen^{\alpha,\beta,\gamma,\delta}\left(4,3,\nu\right)$, such that $(f_{t})_{t \in D_{\epsilon}}$ is a deformation of $f_{0}$ and the algebraic curves $V_{\tau_i}(t)$ are fibers of $f_t$ for all $t$.  Initially we will build the families using two special curves and after we will prove, using some Lemmas that the deformations of the remaining curves are also fibers of  $(f_{t})_{t \in D_{\epsilon}}$.   
 Set $V_{a}=\overline{{f}_{0}^{-1}(a)}$ and $V_{b}=\overline{{f}_{0}^{-1}(b)}$, where $a=[0:1:0:0]$, $b=[1:0:0:0]$, and let by $V_{\tau^\ast}=\overline{{f}_{0}^{-1}(\tau^\ast)}$, where $\tau^\ast \in Sing(\mathcal G_{0})\backslash\{a,b\}$. Using the previous notation, from Lemma \ref{subvariedades} we get $V_{a}(t)$ and $V_{b}(t)$ for all $t \in D_{\epsilon}$. We will use them to define a family of rational maps  $({f}_{t})_{t \in D_{\epsilon}}$, a deformation of $f_{0}$ in $Gen^{\alpha,\beta,\gamma,\delta}\left(4,3,\nu\right)$.

  \begin{proposition}\label{recupmapas}
Let $(\mathcal{F}_{t})_{t \in D_{\epsilon}}$ be a deformation of $\mathcal F_{0}= f_{0}^*(\mathcal G_{0})$, where $(f_{0}, \mathcal G_{0})$ is a generic pair, with $\mathcal G_{0} \in \mathcal M(d,3)$, $f_0 \in Gen^{\alpha,\beta,\gamma,\delta}\left(4,3,\nu\right)$ and $deg(f_{0})=\nu\geq 2$. Then there exists a deformation $({f}_{t})_{t \in D_{\epsilon}}$ of $f_{0}$ in $Gen^{\alpha,\beta,\gamma,\delta}\left(4,3,\nu\right)$ such that:
\begin{enumerate}
\item[(i)]$V_{a}(t)$ and $V_{b}(t)$ are fibers of $({f}_{t})_{t \in D_{\epsilon}}$.
\item[(ii)] $\mathcal C(\mathcal F_t)=I(f_{t}), \forall {t \in D_{\epsilon}}$.
    \end{enumerate}
  \end{proposition}  
  We will explain in detail the situation corresponding to the  case (1), that is, \break when  $1<\alpha<\beta<\gamma<\delta$ and $\alpha$, $\beta$, $\gamma$ and $\delta$ are pairwise relatively prime.   \\
\begin{proof} Let $\tilde f_{0}=(F^\alpha_{0},F^\beta_{1},F^\gamma_{2},F^\delta_{3}): \mathbb C^{5} \to\mathbb C^{4}$ be the homogeneous expression of $f_{0}$. Then $V_{a}$ and $V_{b}$ appear as the complete intersections $\{F_{0}=F_{2}=F_{3}=0\}$ and $\{F_{1}=F_{2}=F_{3}=0\}$ respectively. Hence $I(f_{0})=V_{a} \cap V_{b}$. Using Sernesi's stability criteria (see \cite[section 4.6, p.235-236]{Ser0}), it follows that $V_{a}(t)$ and $V_{b}(t)$ appear as the complete intersections, say $V_{a}(t)=\{F_{0}(t)=F_{2}(t)={F_{3}}(t)=0\}$ and  $V_{b}(t)=\{F_{1}(t)=\widehat{F_{2}}(t)=\widehat{F_{3}}(t)=0\}$ where $F_{i}(t)$ for $0\leq i\leq3$ and $\widehat{F_{i}}(t)$ for $2\leq i\leq3$ are holomorphic deformations of $F_{i}$ and $ D_{\epsilon}$ is a possibly smaller neighborhood of $0$. Let us find polynomials $P_{i}(t)$ for $0\leq i\leq3$ such that $V_{a}(t)=\{P_{0}(t)=P_{2}(t)=P_{3}(t)=0\}$, $V_{b}(t)=\{P_{1}(t)=P_{2}(t)=P_{3}(t)=0\}$. Observe first that since each ${F_{i}}(t)$ is near $F_{i}$ for $0\leq i\leq3$ respectively, they meet as a regular complete intersection at $J(t)=\{\bigcap_{i=0}^{i=3}F_{i}(t)=0\}= V_a(t) \cap \{F_1(t)=0\}.$ Analogously each $\widehat{F_{i}}(t)$ is near $F_{i}$ for $2\leq i\leq3$.
Hence $J(t) \cap \{\hat F_{2}(t)=\hat F_{3}(t)=0\}=V_{b}(t)\cap V_{a}(t)=\mathcal C(\mathcal F_t)$, which implies that $\mathcal C(\mathcal F_t)\subset J(t).$  Once $\mathcal C(\mathcal F_t)$ and $J(t)$ have $\frac{\nu^{4}}{\alpha \beta \gamma\delta}$ points, we have that  $\mathcal C(\mathcal F_t)=J(t)$ for all $t \in D_{\epsilon}$. Using Noether's Theorem (see \cite{cln1}, \cite[p.86]{ln}) and the fact that all polynomials involved are homogeneous, we have $\widehat F_{i}(t)$ for $2\leq i\leq3$ $\in$ $<F_{0}(t),F_{1}(t),F_{2}(t),F_{3}(t)>$. Since  $\widehat {F_{3}}(t)$ has the lowest degree, we can assume that $\widehat {F_{3}}(t)=F_{3}(t)$. Since $deg\{F_{2}(t)\}>deg\{F_{3}(t)\}$ we also have $\hat F_{2}(t)=F_{2}(t)+g_1(t)F_{3}(t)$, where $g_1(t)$ is a homogeneous polynomial of degree $deg\{F_{2}(t)\}-deg\{F_{3}(t)\}$.
Moreover, note that $V_{b}(t)=V\{F_{1}(t),\hat{F_{2}}(t),\hat{F_{3}}(t)\}=V\{F_{1}(t),F_{2}(t)+g_1(t)F_{3}(t),{F_{3}}(t)\}=V\{F_{1}(t),{F_{2}}(t),{F_{3}}(t)\}=0$, where $V\{H_{1},H_{2},H_{3}\}$ denotes the projective algebraic curve defined by $\{H_{1}=H_{2}=H_{3}=0\}$. Hence we can define $f_{t}=\{P_{0}^{\alpha}(t),P_{1}^{\beta}(t),P_{2}^{\gamma}(t),P_{3}^{\delta}(t)\}$ where  $P_{i}(t)=F_{i}(t)$. This defines a  family of mappings $({f}_{t})_{t \in D_{\epsilon}}:\mathbb P^{4} \DashedArrow[->,densely dashed    ]\mathbb P^{3}$, and $V_{a}(t)$ and $V_{b}(t)$ are fibers of ${f}_{t}$ for fixed $t$. Observe that, for ${\epsilon}$ sufficiently small, $({f}_{t})_{t \in D_{\epsilon}}$ is generic in the sense of definition \ref{generic}, and its indeterminacy locus $I({f}_{t})$ is precisely $\mathcal C(\mathcal F_t)$. Furthermore, since $Gen^{\alpha,\beta,\gamma,\delta}\left(4,3,\nu\right)$ is open, we can suppose that this family $({f}_{t})_{t \in D_{\epsilon}}$ is in it.\end{proof}
 \begin{remark} {\rm{Proposition \ref{recupmapas}}} can be adapted with minor modifications to the cases $(2), (3)$ and $(4)$ respectively. \end{remark}
To finish, we have to prove that $V_{\tau^\ast}(t)$, where $\tau^\ast \in Sing(\mathcal G_{0})\backslash\{a,b\}$ are also fibers of $({f}_{t})_{t \in D_{\epsilon}}$. This will be done with the help of two auxiliary Lemmas.

 In the local coordinates $X(t)=(x_{0}(t),x_{1}(t),x_{2}(t),x_{3}(t))$ near some point of $\mathcal C(\mathcal F_t)$ the local expression of the polynomials $P_{i}(t)$ $i=0\leq i\leq3$ that are components of the map ${f}_{t}$ can be written as $P_{i}(t)=u_{it}x_{0}(t)+[\Pi_{j\neq i}x_{j}(t)].h_{it}$ where $u_{it} \in \mathcal O^{*}(\mathbb C^4,0)$ and $h_{it} \in \mathcal O(\mathbb C^4,0)$. Note that for any $i$, $\lim_{t\to 0}h_{it}=0$. 
 We want to show that an orbit of the $\mathbb C^{\ast}$-action
\begin{equation}\label{orbita1}
 (x_0,x_1,x_2,x_3)\to(s^{\beta\gamma\delta}x_0,s^{\alpha\gamma\delta}x_1,s^{\alpha\beta\delta}x_2,s^{\alpha\beta\gamma}x_3)
 \end{equation} that extends globally as a singular curve of the foliation $\mathcal F_{t}$ is a fiber of ${f}_{t}$. An orbit that is not contained in any coordinate plane will called generic orbit.
Recall that the condition $\alpha<\beta<\gamma<\delta$ implies that $\beta(\gamma\delta+\alpha\delta+\alpha\gamma)>\alpha\gamma\delta$ and $\delta(\beta\gamma+\alpha\gamma+\alpha\beta)>\alpha\beta\gamma$.
\begin{Lemma}\label{casomenor} If 
 $\alpha(\gamma\delta+\beta\delta+\beta\gamma)\geq\beta\gamma\delta$  and 
 $\gamma(\beta\delta+\alpha\delta+\alpha\beta)\geq\alpha\beta\delta$, 
 then any generic orbit of the $\mathbb C^{\ast}$-action given by the expression \ref{orbita1} that extends globally as a singular curve of the foliation $\mathcal F_{t}$ is also a fiber of $f_{t}$ for fixed $t$.
\end{Lemma}
In order to simplify we will omit the parameter $t$.
\begin{proof}  A generic orbit $\psi(s)$ can be parametrized as \break$s\to(m_0s^{\beta\gamma\delta},m_1s^{\alpha\gamma\delta},m_2s^{\alpha\beta\delta},m_3s^{\alpha\beta\gamma})$; $m_0m_1m_2m_3\neq0$. Without loss of generality, we can suppose that $m_0=m_1=m_2=m_3=1$. We have $f_{t}(\psi(s))=[(s^{\beta\gamma\delta}u_{0}+s^{\alpha(\gamma\delta+\beta\delta+\beta\gamma)}h_{0})^\alpha:(s^{\alpha\gamma\delta}u_{1}+s^{\beta(\gamma\delta+\alpha\delta+\alpha\gamma)}h_{1})^{\beta}:(s^{\alpha\beta\delta}u_{2}+s^{\gamma(\beta\delta+\alpha\delta+\alpha\beta)}h_{2})^\gamma:(s^{\alpha\beta\gamma}u_{3}+s^{\delta(\beta\gamma+\alpha\beta+\alpha\beta)}h_{3})^\delta].$ The conditions $\alpha(\gamma\delta+\beta\delta+\beta\gamma)\geq\beta\gamma\delta$  and 
 $\gamma(\beta\delta+\alpha\delta+\alpha\beta)\geq\alpha\beta\delta$ enables us to extract  $s^{\alpha\beta\gamma\delta}$ from $f_{t}(\psi(s))$.
Hence we obtain \begin{equation}\label{eq4}f_{t}(\psi(s))=[(u_{0}+s^{k}h_{0})^\alpha:(u_{1}+s^{l}h_{1})^{\beta}:(u_{2}+s^{m}h_{2})^\gamma:(u_{3}+s^{n}h_{3})^\delta]\end{equation}
where $k={\alpha(\gamma\delta+\beta\delta+\beta\gamma)}-{\beta\gamma\delta}$, $l={\beta(\gamma\delta+\alpha\delta+\alpha\gamma)}-{\alpha\gamma\delta}$, \break$m={\gamma(\beta\delta+\alpha\gamma+\alpha\beta)}-{\alpha\beta\delta}$ and $n={\delta(\beta\gamma+\alpha\gamma+\alpha\beta)}-{\alpha\beta\gamma}$.
Since $V_{\tau}$ is a fiber of $f$, $f_{0}(V_{\tau})=[d:e:f:g]\in \mathbb P^3$ with $d.e.f.g\neq0$. If we take a covering of $I(f)=\{p_{1},\dots,p_{j},\dots,p_{\frac{\nu^4}{\alpha\beta\gamma\delta}}\}$ by small open balls $B_{j}(p_{j})$, $1\leq j\leq\frac{\nu^4}{\alpha\beta\gamma\delta}$, the set $V_{\tau}\backslash\cup_{j}B_{j}(p_{j})$ is compact. For a small deformation $f_{t}$ of $f_{0}$ we have that $f_{t}[V_{\tau}(t)\backslash\cup_{j}B_{j}(p_{j})(t)]$ stays near $f[V_{\tau}\backslash\cup_{j}B_{j}(p_{j})]$. Hence for $t$ sufficiently small the components of expression \ref{eq4} do not vanish neither inside nor outside of the neighborhood $\cup_{j}B_{j}(p_{j})(t)$. This is possible only if $f_{t}$ is constant along these curves. 

In fact, $f_{t}(V_{\tau}(t))$ is either a surface, a curve or a point. If it is a surface then it cuts all lines of $\mathbb P^3$ and therefore the components should be zero somewhere. Similarly if it is a curve then it cuts all 2-planes of $\mathbb P^3$ and therefore the components should be zero somewhere. Hence $f_{t}(V_{\tau}(t))$ is constant and we conclude that $V_{\tau}(t)$ is a fiber.  
 \end{proof}
 
 We would like to note that repeating the previous argument we can also conclude that the orbits $\psi_1(s)=(0,0,0,s)$ and $\psi_2(s)=(0,s,0,0)$ that correspond to the intersection $(x_0(t)=x_1(t)=x_2(t)=0)$ and $(x_0(t)=x_2(t)=x_3(t)=0)$
 when they are extended globally as curves in $\mathbb P^4$ are also fibers of $f_t$ for fixed $t.$ 
 
 When $\alpha(\gamma\delta+\beta\delta+\beta\gamma)<\beta\gamma\delta$  and 
 $\gamma(\beta\delta+\alpha\delta+\alpha\beta)<\alpha\beta\delta$
 the situation requires more detail. 
 For our strategy to work we need to write expressions of $P_0(t)$ and $P_2(t)$ in such a way that when we evaluate $f_t$ over a generic orbit  that extends globally as a singular curve of the foliation $\mathcal F_t$ we can repeat the situation as in Lemma \ref{casomenor}. 

 First of all we will suppose that the orbits which are contained in the coordinate planes that extend globally as singular curves of the foliation $\mathcal F_{t}$ are fibers of $f_{t}$. Using this fact we will work with two such orbits, the first contained on the plane $x_{0}(t)=0$ and the second contained on the plane $x_{2}(t)=0$.
 
 The orbit contained in the plane $x_{0}(t)=0$ can be written as $(x_{0}=x_{1}^\beta-c_0x_{2}^\gamma=x_{1}^\beta-c_1x_{3}^\delta=0)$ and analogously the orbit contained in the plane $x_2(t)=0$ can be written as $(x_{2}=x_{0}^\alpha-c_2x_{1}^\beta=x_{0}^\alpha-c_3x_{3}^\delta=0)$. We have that the germ of $P_{0}(t)$ at the point $p_{j}(t)\in\mathcal C(\mathcal F_t)$ belongs to the ideal generated by $x_{0}(t)$, $(x_{1}^\beta-c_0x_{2}^\gamma)(t)$ and $(x_{1}^\beta-c_1x_{3}^\delta)(t)$. Hence we can write the function $h_{0t}$ from the expression $P_{0}(t)=u_{0t}x_{0}(t)+x_{1}(t)x_{2}(t)x_{3}(t)h_{0t}$ as
$$h_{0t}=x_{0}(t)h_{01t}+(x_{1}^\beta(t)-c_0x_{2}^\gamma(t))h_{02t}+(x_{1}^\beta-c_1x_{3}^\delta(t))h_{03t}$$
where $h_{j1t},  \in \mathcal O^{*}(\mathbb C^4,0)$, where $1\leq j\leq 3$.  In the same way we have that the germ of $P_{2}(t)$ at the point $p_{j}(t)\in\mathcal C(\mathcal F_t)$ belongs to the ideal generated by $x_{2}(t)$, $(x_{0}^\alpha-c_2x_{1}^\beta)(t)$ and $(x_{0}^\alpha-c_3x_{3}^\delta)(t)$. Hence we can write the function $h_{2t}$ from the expression $P_{2}(t)=u_{2t}x_{2}(t)+x_{0}(t)x_{1}(t)x_{3}(t)h_{2t}$ as
$$h_{2t}=x_{2}(t)h_{21t}+(x_{0}^\alpha-c_2x_{1}^\beta)(t))h_{22t}+(x_{0}^\alpha-c_3x_{3}^\delta)(t))h_{23t}$$
where $h_{j2t},  \in \mathcal O^{*}(\mathbb C^4,0)$, where $1\leq j\leq 3$. In this way,  we can repeat the argument of Lemma \ref{casomenor} and extract the factor $s^{\alpha\beta\gamma\delta}$ and the result follows. Let us now prove that the orbits contained in the coordinate planes that extend globally as singular curves of the foliation $\mathcal F_{t}$ are fibers of $f_{t}$.  
\begin{Lemma}  \label{casomaior}  If (i) $\alpha(\gamma\delta+\beta\delta+\beta\gamma)<\beta\gamma\delta$  and 
 (ii) $\gamma(\beta\delta+\alpha\delta+\alpha\beta)<\alpha\beta\delta$
  then any orbit of the $\mathbb C^{\ast}$-action given by the expression \ref{orbita1} which is contained in the coordinate planes $x_0(t)=0$ and $x_2(t)=0$ at $p_{j}(t)\in\mathcal C(\mathcal F_t)$ and that extends globally as a singular curve of the foliation $\mathcal F_{t}$ is a fiber of the mapping $f_{t}$ for fixed $t$. 
\end{Lemma}
We shall prove only the condition $(i)$, which corresponds to the orbit contained in the plane $x_0(t)=0$ because the condition $(ii)$ is analogous. We observe that in the last case we have to work with the orbit contained in plane $x_2(t)=0$.   
To simplify the notation we will omit the index $t$ in some expressions.
 The idea is to prove firstly that $f_{t}(V_{\tau}(t))$ is contained in a line in $\mathbb P^3$ and then to prove that $f_{t}(V_{\tau}(t))$ is in fact a unique point using degree theory and the Riemann-Hurwitz formula. 
\begin{proof}   We can suppose that such an orbit can be parametrized as $s\to(0,s^{\gamma\delta},s^{\beta\delta},s^{\beta\gamma})$.  After evaluating the mapping $f_{t}$ on this orbit we get:
$$f_{t}(\psi(s))=[s^{\alpha(\gamma\delta+\beta\delta+\beta\gamma)}h_{0}^\alpha:s^{\beta\gamma\delta}u_1^{\beta}:s^{\beta\gamma\delta}u_2^{\gamma}:s^{\beta\gamma\delta}u_3^{\delta}].$$ 
This can be written as 
\begin{equation}\label{eq5}
[s^{\alpha(\gamma\delta+\beta\delta+\beta\gamma)}\tilde{h_{0}}:s^{\beta\gamma\delta}u_1^{\beta}:s^{\beta\gamma\delta}u_2^{\gamma}:s^{\beta\gamma\delta}u_3^{\delta}]=[X(s):Y(s):Z(s):W(s)].
\end{equation} 

First we prove that $f_{t}(V_{\tau}(t))$ is contained in a line of the form $(W-\lambda_1Z=Z-\lambda_2 Y=0)$ of $\mathbb P^3.$
Let us consider the meromorphic function with values in $\mathbb{P}^1$ given by $g_{1t}(s)=\frac{Z(s)}{Y(s)}=\frac{u_2^{\gamma}}{u_1^{\beta}}$. When $s\to0$ this function goes to a constant $\lambda_2\neq0,\lambda_2\neq\infty$.
Observe that  for small $t$ the function $\frac{P_{1}^{\beta}}{P_{2}^{\gamma}}(t): V_{\tau}(t)\backslash\cup_{j}B_{j}(p_{j}(t)) \rightarrow \mathbb{P}^1$ stays near $\frac{P_{1}^{\beta}}{P_{2}^{\gamma}}(0): V_{\tau}(0)\backslash\cup_{j}B_{j}(p_{j}(0)) \rightarrow \mathbb{P}^1$. Note that since $V_{\tau}(0)$ is a fiber $\frac{P_{1}^{\beta}}{P_{2}^{\gamma}}(0)$ does not vanish. We conclude that $f_{t}(V_{\tau}(t))\subset(Z-\lambda_2 Y=0)$. Analogously, $f_{t}(V_{\tau}(t))\subset(W-\lambda_1 Z=0)$. Therefore we conclude that that $f_{t}(V_{\tau}(t))$ is contained in a line of the form $(W-\lambda_1Z=Z-\lambda_2 Y=0)$. If $\beta(\gamma\delta+\alpha\delta+\alpha\gamma)>\alpha\gamma\delta$ we can write equation \ref{eq5} as $$[\tilde{h}_{0}(s):s^m u_1^{\beta}:s^m u_2^{\gamma}:s^m u_3^{\delta}]$$ where $m=\beta\gamma\delta-\beta(\gamma\delta+\alpha\delta+\alpha\gamma)$. Observe that when $s=0$ the function  $\tilde{h}_{0}(s)$ could vanish; in this case such a point corresponds to a indeterminacy point $p_{j}(t)$ of $f_{t}$ for some $j$.  At $p_{j}(t)$ we can write the first component of equation \ref{eq5} as $\tilde{h}_{0}(s)=s^{\rho_j}\tilde{h}_{j}(s)$ where either $\tilde{h}_{j}(s) \in \mathcal O^{\ast}(\mathbb C,0)$ or $\tilde{h}_{0}\equiv0$. However, in the second case we are done, that is, $V_{\tau}(t)$ is a fiber of $f_{t}$. At $p_{j}(t)$ we have two possibilities:
Case (1): ${\rho_j}< m$. In this case we can write equation \ref{eq5} as
 \begin{equation}\label{eq6}
[\tilde{h}_{j}(s):s^{m-{\rho_j}} u_1^{\beta}:s^{m-{\rho_j}} u_2^{\gamma}:s^{m-{\rho_j}} u_3^{\delta}].
\end{equation}
If $s\to0$ the image goes to $[1:0:0:0]$, hence ${f_{t}}|_{V_{\tau}(t)}(p_{j}(t))=[1:0:0:0]$.

Case (2): ${\rho_j}\geq m$. We can write equation \ref{eq6} as
\begin{equation}\label{eq7}
[s^{{\rho_j}-m}\tilde{h}_{j}(s):  u_1^{\beta}: u_2^{\gamma}:u_3^{\delta}].
\end{equation} If $s\to0$ the image goes to $[a:1:\lambda_2:\lambda_1\lambda_2]$ where $a\in\mathbb C$. This is due to the fact that the image of such a point belongs to the curve $(W-\lambda_1Z=Z-\lambda_2 Y=0)\simeq\mathbb P^1$ and hence we can write it as $[a:1:\lambda_2:\lambda_1\lambda_2]$. Suppose that $f_{t}|_{V_{\tau}(t)}$ is not constant and consider the mapping $f_{t}|_{V_{\tau}(t)}:V_{\tau}(t) \to f_{t}(V_{\tau}(t))\subset(W-\lambda_1Z=Z-\lambda_2 Y=0)$ for fixed $t$. Let $Q=\{j|\rho_{j}<m\}$. Note that $p\in V_{\tau}(t)$ and $f_{t}|_{V_{\tau}(t)}(p)=[1:0:0:0]$ imply that $p=p_j(t)$ for some $j \in Q$; that is, $(f_{t}|_{V_{\tau}(t)})^{-1}[1:0:0:0]=\{p_{j}(t),j \in Q\}$. Moreover, by equation \ref{eq7} we have $mult(f_{t}|_{V_{\tau}(t)},p_j(t))=m-\rho_j$. In particular, the degree of $f_{t}|_{V_{\tau}(t)}$ is $$deg(f_{t}|_{V_{\tau}(t)})=\sum_j(m-\rho_j).$$ On the other hand, if $p\in(f_{t}|_{V_{\tau}(t)})^{-1}[0:1:\lambda_2:\lambda_1\lambda_2]$ then $(P_{0}^{\alpha}(p)=0)$ and so $mult(f_{t}|_{V_{\tau}(t)},p)$ is equal to the intersection number of $(P_{0}^{\alpha}(t)=0)$ and $V_\tau(t)$ at $p$. Hence 
$$deg(f_{t}|_{V_{\tau}(t)})={V_{\tau}(t)}.P_{0}^{\alpha}(t)=deg({V_{\tau}(t)})\times deg(P_{0}^{\alpha}(t))=\frac{\nu^4}{\alpha}=\sum_j(m-\rho_j).$$ 
But $(m-\rho_j)\leq m=\beta\gamma\delta-\beta(\gamma\delta+\alpha\delta+\alpha\gamma)$ and so 
  $$\sum_{j\in Q}(m-\rho_j)\leq\# Q\times m\leq\frac{\nu^4}{\alpha\beta\gamma\delta}\times (\beta\gamma\delta-\beta(\gamma\delta+\alpha\delta+\alpha\gamma))={\nu^4}(\frac{1}{\alpha}-\frac{1}{\beta}-\frac{1}{\gamma}-\frac{1}{\delta})$$
  which implies that $\frac{1}{\alpha}\leq\frac{1}{\alpha}-\frac{1}{\beta}-\frac{1}{\gamma}-\frac{1}{\delta}$ and we arrive at a contradiction. Therefore, $Q=\emptyset$, $f_{t}|_{V_{\tau}(t)}$ is a constant and ${V_{\tau}(t)}$ is a fiber of $f_{t}$. \end{proof}

\begin{remark} Now we will discuss the minor modifications that we need to do obtain analogous results for the remaining cases
\begin{itemize}
\item For the case $(2)$ just replace $\alpha$ by $1$ and repeat the same argument. 
\item Cases $(3)$ and $(4)$ are simpler.  An analogue of Lemma \ref{casomenor} is enough to prove that any orbit (generic or which is contained in any coordinate plane) that extends as a singular curve of the foliation $\mathcal F_t$ is a fiber of $f_t$.  
\end{itemize}
\end{remark}
\subsection{Construction of the families $(\mathcal G_{t})_{t \in D_{\epsilon}}$}\label{familia das folheacoes} We will restrict to the case (1) because the other cases are analogous.

In this part we will explain how to obtain a family of candidates to be a deformation of $\mathcal G_0$. \par Set $\mathcal C(\mathcal F_0)=\{p_{1},\dots,p_{j},\dots,p_{\frac{\nu^4}{\alpha\beta\gamma\delta}}\}$. According to \cite[Theorem 3, p. 653]{ln3}, for each $j\in\{1,\dots,\frac{\nu^4}{\alpha\beta\gamma\delta}\}$ there is:

\begin{itemize}
\item[a)] a germ of a holomorphic one parameter family of foliations $t\in D_{\epsilon}\mapsto \mathcal G^{j}_t=(\eta^{j}_t=0)$ where $(\eta^{j}_t)_{t \in D_{\epsilon}}$ is a $2$-dimensional  quasi-homogeneous singularity of type\break $(\beta\gamma\delta,\alpha\gamma\delta,\alpha\beta\delta,\alpha\beta\gamma;\alpha\beta\gamma\delta(d-1))$. 

\item[b)] a holomorphic germ of curve $p_j: D_{\epsilon} \to \mathbb P^{4}$ in such a way that 
for each fixed $t \in D_{\epsilon}$,  $(\eta^{j}_t)(p_j(t))=0$.
\end{itemize} 

In other words, the persistence of the quasi-homogeneous singular set of $\mathcal F_0$ gives us (1) a holomorphic path of singular points and (2) a holomorphic path of integrable $1$-forms of the same type of the of original foliation.  

Moreover, each $\eta^{j}_t$ is invariant under the $\mathbb C^\ast$-action \begin{equation}
(x_0,x_1,x_2,x_3)\to(s^{\beta\gamma\delta}x_0,s^{\alpha\gamma\delta}x_1,s^{\alpha\beta\delta}x_2,s^{\alpha\beta\gamma}x_3).
\end{equation}
As a consequence the $1$-forms $(\eta^{j}_t)_{t \in D_{\epsilon}}$ naturally define a $1$-dimensional foliation on $\mathbb P^{3}_{w}$, the weighted projective $3$-space with weights $w=(\beta\gamma\delta,\alpha\gamma\delta,\alpha\beta\delta,\alpha\beta\gamma)$.  On the other hand, since $1<\alpha<\beta<\gamma<\delta$ and $\alpha,\beta,\gamma$ and $\delta$  are pairwise relatively prime, we can conclude using \cite[Lemma 5.7, p.106]{Fletcher} that $\mathbb P^{3}_{w}$ is biholomorphic to $\mathbb P^{3}$. Hence, any family $(\eta^{j}_t)_{t \in D_{\epsilon}}$ can be interpreted as a family of $1$-dimensional foliations on $\mathbb P^3$. 
It remanis to  see that these candidates leave the coordinate planes invariant.
 Now we will prove a Lemma which ensures that the deformations $(\eta^{j}_t)_{t \in D_{\epsilon}}$ of $\eta^{j}_0$ also have invariant hyperplanes. Recall that $\eta^{j}_0=\eta$ as in the equation $\ref{eta}.$ 
 
 \begin{Lemma} Let  $(\eta^{j}_t)_{t \in D_{\epsilon}}$ be a holomorphic deformation of $\eta^{j}_0$, where $\eta^{j}_0$ is a $2$-dimensional  quasi-homogeneous singularity of type $(\beta\gamma\delta,\alpha\gamma\delta,\alpha\beta\delta,\alpha\beta\gamma;\alpha\beta\gamma\delta(d-1))$ of $\mathcal F_0$. If $\alpha$, $\beta$, $\gamma$ and $\delta$ are as:
 
\begin{itemize}
\item In case $(1)$ above the deformation $\eta^{j}_t$ of $\eta_0$ leaves invariant the coordinate hyperplanes $(x_0x_1x_2x_3=0)$, 
\item In case $(2)$ above the deformation $\eta^{j}_t$ of $\eta_0$ leaves invariant the coordinate hyperplanes $(x_1x_2x_3=0)$,
\item In case $(3)$ above the deformation $\eta^{j}_t$ of $\eta_0$ leaves invariant the coordinate hyperplanes $(x_2x_3=0)$, 
\item In case $(4)$ above the deformation $\eta^{j}_t$ of $\eta_0$ leaves invariant the coordinate hyperplane $(x_3=0)$. 
\end{itemize}
\end{Lemma}

We will prove only the case $(1)$, for the other cases the argument is analogous. 
\begin{proof} 
Case $(1)$:
The holonomy map (see \cite[section 6.6.1, p.129]{cacede}) of the $x_3$-axis at $x_3=1$ is $$H(x_0,x_1,x_2)=(e^{2i\pi \frac{\delta}{\alpha}}.x_0,e^{2i\pi \frac{\delta}{\beta}}.x_1,e^{2i\pi \frac{\delta}{\gamma}}.x_2).$$ We will prove that this holonomy map leaves invariant the foliation $\eta^{j}_t|_{(x_3=1)}$. Let us write the foliation $\eta_{t}^{j}|_{(x_3=1)}$ as a vector field $Y_t^{j}$, $\eta^{j}_t|_{(x_3=1)}=i_{Y_t^{j}}(dx_0\wedge dx_1\wedge dx_2)$.

When $t=0$ we have that
\begin{multline}\label{eta1}
 \eta_{0}^{j}|_{(x_3=1)}=\alpha\beta x_2P_{23}(x_0^{\alpha},x_1^{\beta},x_2^{\gamma},1)dx_0\wedge dx_1\\-\alpha\gamma x_1P_{13}(x_0^{\alpha},x_1^{\beta},x_2^{\gamma},1)dx_0\wedge dx_2\\+\beta\gamma x_0P_{03}(x_0^{\alpha},x_1^{\beta},x_2^{\gamma},1)dx_1\wedge dx_2
 \end{multline}
  and so 
  
  \begin{multline}\label{Y0}
  Y_0^{j}=\beta\gamma x_0P_{03}(x_0^{\alpha},x_1^{\beta},x_2^{\gamma},1)\frac{\partial}{\partial x_0}-\alpha\gamma x_1P_{13}(x_0^{\alpha},x_1^{\beta},x_2^{\gamma},1)\frac{\partial}{\partial x_1}\\+\alpha\beta x_2P_{23}(x_0^{\alpha},x_1^{\beta},x_2^{\gamma},1)\frac{\partial}{\partial x_2},
   \end{multline} so that $H^{\ast}Y_0^{j}=Y_0^{j}$. It follows that $H^{\ast}Y_t^{j}=Y_t^{j}$. Let us prove that the axis $(x_0=0)$ is $Y_t^{j}$-invariant. If not, then the first component of $Y_t^{j}$ has a monomial of the type $x_1^{n}x_2^{m}\frac{\partial}{\partial x_0}$, for which $H^{\ast}(x_1^{n}x_2^{m}\frac{\partial}{\partial x_0})=e^{2i\pi{(\frac{n\delta}{\beta}+\frac{m\delta}{\gamma}-\frac{\delta}{\alpha})}}(x_1^{n}x_2^{m}\frac{\partial}{\partial x_0})$, where the pair $(n,m)\in\{ \mathbb N^2\}\cup\{\{0\}\times \mathbb N\}\cup\{ \mathbb N\times\{0\}\}$. But this would imply that   
$(\frac{n\delta}{\beta}+\frac{m\delta}{\gamma}-\frac{\delta}{\alpha})\in \mathbb Z$ which is impossible, because  $\alpha$, $\beta$, $\gamma$ and $\delta$ are pairwise relatively prime. Similarly, the $2$-planes $(x_1=0)$, $(x_2=0)$ and 
$(x_3=0)$ are $\eta^{j}_t$ invariant.  For the other axes we proceed in a analogous way. 
\end{proof} 
Therefore, we see that we can push forward the foliation to $\mathbb P^3$. 
 In this way, we get a family of $1$-dimensional foliations of degree $d$, all of them leaving invariant the appropriate $2$-planes for each situation. These foliations will be the candidates to be a deformation of $\mathcal G_{0}$.\\
 {\bf Notation.} We will use the notation:  $\mathcal C(\mathcal F_t)=\{p_{1}(t),\dots,p_{j}(t),\dots,p_{\frac{\nu^4}{\alpha\beta\gamma\delta}}(t)\}$
  
  \begin{remark} We would like to observe that when we make a weighted blow-up with weights $w=(\beta\gamma\delta,\alpha\gamma\delta,\alpha\beta\delta,\alpha\beta\gamma)$ at the point $p_{j}(t)$ the foliation $\mathcal G_{t}^{j} \in  \mathcal M(d,3)$ appears as foliation on the exceptional divisor of the blow up. Indeed, if we denote the blow up  $\pi_{w}(t):(\widetilde{\mathbb P^4(t)},E_{j}(t))\to (\mathbb P^4,p_j(t))$ then the divisor $E_{j}(t)$ is biholomorphic to $\mathbb P^3$ and $\pi_{w}(t)^{\ast}(\mathcal F_t)$ extends to $\widetilde{\mathbb P^4(t)}$, the complex orbifold obtained after the blow-up \cite{mamor}. 
     We would like to observe that the new space $(\widetilde{\mathbb P^4(t)},E_{j}(t))$ has four 2-planes (each of them biholomorphic to $\mathbb P^2$) of singular points located at the exceptional divisor $E_{j}(t)=\mathbb P^3$. They actually coincide with the four coordinate $2$-planes of $\mathbb P^{3}_{w}$. Recall that at $p_{j}(t)$, the vector field $$S={\beta\gamma\delta}\frac{\partial}{\partial x_0}+{\alpha\gamma\delta}\frac{\partial}{\partial x_1}+{\alpha\beta\delta}\frac{\partial}{\partial x_2}+{\alpha\beta\gamma}\frac{\partial}{\partial x_3}$$   is tangent to $\mathcal F_t$ and therefore its strict transform by $\pi_{w}(t)$ is transverse to $E_{j}(t)$. In fact we can summarize all this discussion saying that $\pi_{w}(t)^{\ast}|_{E_{j}(t)}\simeq\mathcal G_{t}^{j} \in  \mathcal M(d,3)$.
 \end{remark}

\begin{remark} In principle, we have $\frac{\nu^4}{\alpha\beta\gamma\delta}$ different families of one dimensional foliations, $(\mathcal G^{j}_t)_{t \in D_{\epsilon}}$, $1\leq j\leq \frac{\nu^4}{\alpha\beta\gamma\delta}$. We cannot assert a priori that, if $i\neq j$, $(\mathcal G^{i}_t)$ is equivalent to $(\mathcal G^{j}_t)$ for any $t\in D_{\epsilon}$. Indeed, this fact is true, but it will be a consequence of the final result.

Since $\mathcal G_0\in  \mathcal M(d,3)$ and $\mathcal M(d,3)$ is Zariski generic, there exists a countable subset $C\subset D_{\epsilon}$ such that  $\mathcal G_t\in  \mathcal M(d,3)$ for all $t\notin C$.
\end{remark}

We will choose a family $(\mathcal G_t)_{t \in D_{\epsilon}}$ as being the family $(\mathcal G^{1}_t)_{t \in D_{\epsilon}}$.
\subsection{End of the proof of Theorem \ref{teob}}\label{ss:33} 
Let $(\mathcal F_t)_{t \in D_{\epsilon}}$ be the germ of deformation of $\mathcal F_0=f_0^*(\mathcal G_0)$ of \S  \ref{plano da prova} and $(f_t,\mathcal G_t)_{t \in D_{\epsilon}}$ be the germ of deformation of $(f_0,\mathcal G_0)$ obtained in \S \ref{recupmapas} and \S \ref{familia das folheacoes}. Since the pair $(f_0,\mathcal G_0)$ is generic and the set of generic pairs is open, the pair $(f_t,\mathcal G_t)$ is also generic for all  ${t \in D_{\epsilon}}$. Consider the holomorphic family of foliations $\left(\widetilde{\mathcal F_t}\right)_{t \in D_{\epsilon}}$ defined by $\widetilde{\mathcal F}_t=f_t^*(\mathcal G_t)$, $\forall\,{t \in D_{\epsilon}}$. Of course $\widetilde{\mathcal F_t}={\mathcal F_0}$ and $\widetilde{\mathcal F}_t\in PB(\Theta_{\nu,d}^{\alpha,\beta,\gamma,\delta},2,4)$, ${t \in D_{\epsilon}}$.

\begin{Lemma}\label{l:36}
${\mathcal F}_t=\widetilde{\mathcal F}_t$ for all $t \in D_{\epsilon}$. In particular, ${\mathcal F}_t\in PB(\Theta_{\nu,d}^{\alpha,\beta,\gamma,\delta},2,4)$, $\forall t \in D_{\epsilon}$.
\end{Lemma}

The idea is to prove that $\widetilde{\mathcal F}_t$ and ${\mathcal F}_t$ have a common leaf $L_t$, for any $t \in D_{\epsilon}$. In particular, the foliations $\widetilde{\mathcal F}_t$ and ${\mathcal F}_t$ coincide in the Zariski closure $\overline{L}_t^{\,Z}$ of $L_t$.  Recall that there exists a germ of countable set $C\subset D_{\epsilon}$ such that $\mathcal G_t\in \mathcal M(d,3)$ for all $t\notin C$. 

The fact that $\mathcal G_t\in\mathcal M(d,3)$ implies that the Zariski closure of any leaf of $\mathcal G_t$ different to $(x_0x_1x_2x_3=0)$ is the whole $\mathbb P^{3}$. As we will see, this will imply that $\overline{L}_t^{\,Z}$ is the whole $\mathbb P^4$, $\forall\,t\notin C$. Since $C$ is countable this will finish the proof of Lemma \ref{l:36}.
\vskip.1in

We will restrict the prove to case $1$, because the proofs follow the same lines, thus we focus on the first for brevity.  

\vskip.1in
We begin with a punctual weighted blowing-up with weights $w=(\beta\gamma\delta,\alpha\gamma\delta,\alpha\beta\delta,\alpha\beta\gamma)$ at the $\frac{\nu^4}{\alpha\beta\gamma\delta}$ points $p_1(t),\dots,p_\frac{\nu^4}{\alpha\beta\gamma\delta}(t)$ of $C(\mathcal F_t)$. Let us denote by $M_{w}(t)$ the complex orbifold obtained from this blow-up procedure, by $\pi_{w}(t)\colon M_{w}(t)\to \mathbb P^4$ the blow-up map and by $E_1(t),...,E_\frac{\nu^4}{\alpha\beta\gamma\delta}(t)$ the exceptional divisors obtained, where $\pi_{w}(t)(E_j(t))=p_j(t)$, $1\le j\le \frac{\nu^4}{\alpha\beta\gamma\delta}$. Each $E_j(t)$ is a weighted projective $3$-space, $\mathbb P^3_{w}$ where $w=(\beta\gamma\delta,\alpha\gamma\delta,\alpha\beta\delta,\alpha\beta\gamma)$. Now use the hypothesis that $1<\alpha<\beta<\gamma<\delta$ in such a way that $\alpha$, $\beta$, $\gamma$ and $\delta$ are pairwise relatively prime. This ensures that $\mathbb P^3_{w}$ is biholomorphic to $\mathbb P^3$ \cite[Lemma 5.7, p.106]{Fletcher}.  A similar construction can be found in details in \cite[Example 3.6, p.957]{mamor}. Denote by $V_{\tau}'(t)$ the strict transform of $V_{\tau}(t)$ by $\pi_{w}(t)$.

\begin{remark}\label{r:37}
{\rm Since the pair $(f_t,\mathcal G_t)$ is generic and $I(f_t)=C(\mathcal F_t)=\{p_1(t),...,p_\frac{\nu^4}{\alpha\beta\gamma\delta}(t)\}$, we can assert the following facts:
\begin{itemize}
\item[(I).] The map $f_t\circ\pi_w(t)$ extends to a holomorphic map $f_t'\colon M_{w}(t)\to\mathbb P^{3}$.

If $\tau$ is a singulirity of $\mathcal G$ such that $\tau\notin(x_0x_1x_2x_3=0)$  then:
\item[(II).] The fiber $f_t'^{-1}(\tau)$, $\tau\in\ \mathbb P ^{3}$ is the strict transform of $f_t^{-1}(\tau)$ by $\pi_w(t)$. It is smooth near $E_j(t)$ and cuts $E_j(t)$ transversely in just one point, for all $1\le j\le \frac{\nu^4}{\alpha\beta\gamma\delta}$.
\item[(III).] $V_{\tau}'(t)$ is a smooth curve and $f_t'$ is a submersion in some neighborhood of $V_{\tau}'(t)$, for all $1\le j\le  \frac{\nu^4}{\alpha\beta\gamma\delta}$.
\end{itemize}
Assertion (I) follows from the fact that the weighted blowing-up solves completely the indeterminacy set of $f_t$ near each $p_j(t)$, $1\le j\le \frac{\nu^4}{\alpha\beta\gamma\delta}$, which is a consequence of the Lemma \ref{casomenor} and Lemma \ref{casomaior}. To extend $f_t'$ on a neighborhood of the singular locus of each one of the $\mathbb P^3_{w}$ we use Hartogs Theorem \cite{fisher}, since these singularities are quotients of actions by finite groups. Of course, (I) $\implies$ (II) $\implies$ (III).} In fact in what follows we will always avoid working over the singular locus of  $M_{w}(t)$ because it is not extremely relevant to our arguments. 
\end{remark}

Let us denote by $\mathcal F_t'$ and $\widetilde{\mathcal F}_t'$ the strict transforms by $\pi_{w}$ of the foliations ${\mathcal F}_t$ and $\widetilde{\mathcal F}_t$, respectively. Note that they are dicritical, that is, each exceptional divisor is not invariant by $\mathcal F_t'$ and $\widetilde{\mathcal F}_t'$.   
Observe that, for each $1\le j\le \frac{\nu^4}{\alpha\beta\gamma\delta}$, the foliation $\widetilde{\mathcal F}_t'|_{E_j(t)}$ is a foliation by curves on $E_j(t)\simeq\mathbb P^{3}$ equivalent to $\mathcal G_t$.
Similarly, ${\mathcal F}_t'|_{E_j(t)}$ is a foliation by curves on $E_j(t)$, but we cannot assert that ${\mathcal F}_t'|_{E_i(t)}$ is equivalent to
${\mathcal F}_t|_{E_j(t)}$ if $i\ne j$. However, by the choice $\mathcal G_t$, we can assert that $\mathcal G_t$ is equivalent to ${\mathcal F}_t'|_{E_1(t)}$.
On the other hand, by using (II) we can define a holomorphic map $\Phi_t\colon M_w(t)\to E_1(t)$, by
\[
\Phi_t(q):=f_t'^{-1}(f_t'(q))\cap E_1(t)\,\,,\,\,\forall\,q\in\,\mathbb P^4\,\,.
\]
Note that the fibers of $\Phi_t$ coincide with the fibers of $f_t'$. In fact, the maps $\Phi_t$ and $f_t'$ are equivalent, in the sense that there exists a biholomorphism $h\colon \mathbb P^{3}\to E_1(t)$ such that $\Phi_t=h\circ f_t'$. In particular, identifying $\mathcal G_t$ with $\mathcal F_t'|_{E_1(t)}$ we can assert that
\[
\widetilde{\mathcal F}'_t=\Phi_t^*(\mathcal G_t)\,\,.
\]

Now, we fix a singularity of $\mathcal G_t$, say $\tau_1(t)$, such that $\tau_1(t)\notin(x_0x_1x_2x_3=0)$ with $V_{\tau_1}'(t)=\Phi_t^{-1}(\tau_1(t))$. Since $\mathcal G_t\in\mathcal M(d,3)$ it has $3$ analytic separatrices through $\tau_1(t)$, all smooth, say $\lambda_1(t),\lambda_2(t)$ and $\lambda_3(t)$, and no other local analytic separatrix.

 Each separatrix $\lambda_j(t)$ is a germ of a complex curve through $\tau(t)$ such that $\lambda_j(t)\setminus\{\tau(t)\}$ is contained in some leaf of $\mathcal G_t$. Moreover, since $\mathcal G_t\in\mathcal M(d,3)$ then its Zariski closure $\overline{\lambda_j(t)}^Z$ is $E_1(t)=\mathbb P^{3}$, because $\mathcal G_t$ has no proper algebraic invariant subset of positive dimension different to the $2$-planes $(x_0x_1x_2x_3=0)$  and also the lines $(x_i=x_j=0)$ for $i\neq j$.

We fix one of these separatrices, say $\lambda_1(t)$. By construction the set $\Phi_t^{-1}(\lambda_1(t))$ satisfies the following property:
\begin{itemize}
\item[1.] It is $\widetilde{\mathcal F}_t'$-invariant. In other words, $V_{\tau_1}'(t)\subset\Phi^{-1}_t({\lambda_1}(t))$ and $\Phi^{-1}_t({\lambda_1}(t))\setminus V_{\tau_1}'(t)$ is an open subset of some leaf of $\widetilde{\mathcal F}_t'$.
\end{itemize}
We can assert also that
\begin{itemize}
\item[2.] If $\mathcal G_t\in\mathcal M(d,3)$ and $\tau_1(t)$ was chosen as previous then the Zariski closure $\overline{\Phi^{-1}_t(\lambda_1(t))}^{\,Z}$ is $\mathbb P^4$. This follows from the relation
\[
\Phi_t^{-1}\left(\overline{\lambda_1(t)}^{\,Z}\right)=\overline{\Phi_t^{-1}(\lambda_1(t))}^{\,Z}\,\,.
\]
\end{itemize}

{\bf Notation.} A  $\ell$-dimensional {{\it strip} } along $V_{\tau_1}'(t)$, $\ell\geq 2$ is a germ of smooth complex manifold of dimension $\ell$ along $V_{\tau_1}'(t)$, containing $V_{\tau_1}'(t)$ and transverse to $E_1(t)$. 
We say that the strip $\Gamma$ along $V_{\tau_1}'(t)$ is a separatrix of $\mathcal F_t'$ (resp. $\widetilde{\mathcal F}_t'$) along $V_{\tau_1}'(t)$ if it is $2$-dimensional and $\Gamma\setminus V_{\tau_1}'(t)$ contained in some leaf of $\mathcal F_t'$ (resp. $\widetilde{\mathcal F}_t'$). If $\Gamma$ is a separatrix of $\mathcal F_t'$ (or of $\widetilde{\mathcal F}_t'$) along $V_{\tau_1}'(t)$ then $\Gamma\cap E_1(t)$ is one of the separatrices $\lambda_j(t), 1\leq j\leq 3$, of $\mathcal G_t$.
We say that the strip is $\Phi_t$-invariant if it is a union of fibers of $\Phi_t$, or equivalently $\Phi_t^{-1}(\Gamma\cap E_1(t))=\Gamma$.

\begin{claim}\label{cl:38}
Let $\Gamma$ be a $2$-dimensional strip along $V_{\tau_1}'(t)$. Assume that there exists a $3$-dimensional strip $\Sigma$ along 
$V_{\tau_1}'(t)$ such that: 
\begin{itemize}
\item[(a).]  $\Sigma$ is $\Phi_t$-invariant. 
\item[(b).] $\Gamma$ and $\Sigma$ are transverse and $\Gamma\cap\Sigma=V_{\tau_1}'(t)$.
\end{itemize} 

Then $\Gamma$ is $\Phi_t$-invariant. In particular, if $\Gamma$ is a separatrix of $\widetilde{\mathcal F}_t'$ then it is also a separatrix of $\widetilde{\mathcal F}_t'$.

\end{claim}

\begin{proof}
Consider a representatives of $\Gamma$ and $\Sigma$ transverse to $E_1(t)$, denoted by the same symbols. Since $\Phi_t$ is a submersion at the points of $V_{\tau_1}'(t)=\Phi_t^{-1}({\tau_1}(t))$ and by (b), up to a translation, there exists a holomorphic coordinate system around ${\tau_1}(t)\in E_1(t)\subset M_w(t)$, say $(x,y)\colon U\to\C^{3}\times\C$, $x=(x_1,x_2,x_{3})$, such that
\begin{itemize}
\item[(i).] $x({\tau_1}(t))=0\in\C^{3}$ and $y({\tau_1}(t))=0\in\C$.
\item[(ii).] $E_1(t)\cap U=(y=0)$ and $V_{\tau_1}'(t)\cap U=(x=0)$.
\item[(iii).] $\Phi_t(x,y)=(x,0)$.
\item[(iv).] $\Gamma\cap E_1(t)\cap U=(x_2=x_3=y=0)$.
\item[(v).] $\Sigma\cap E_1(t)\cap U\subset(x_1=y=0)$.

Fix a polydisc $Q=\{(x,0)|\, |x_j<\epsilon, 1\leq j\leq3\}\subset U\subset E_1(t)$. We can take the representatives $\Sigma$ and $\Gamma$ in such a way that
\item[(vi).] $\Sigma\cup\Gamma\subset\Phi_t^{-1}(Q)$.
\end{itemize}

Let us denote  $p(s):=(s,0,0,0)\in\Gamma\cap E_1(t)\subset Q$. We assert that there exists $0<\delta<\epsilon$ such that if $|s|<\delta$ then $\Phi_t^{-1}(p(s))\subset \Gamma$.  Note that this implies Claim \ref{cl:38} because $\Phi_t^{-1}(p(s))| |s|<\delta\})$ is open subset of $\Gamma$.
 \par Let us prove the assertion. Given $s\in \mathbb C$ with $|s|<\epsilon$ set  $\sigma_s=\{(x,0)\in Q|\,x_1=s\}$ and $\Sigma_s:=\Phi_t^{-1}(\sigma_s)$, so that $\Sigma=\Sigma_0$.
Note that $\sigma_s$ is an hypersurface of $E_1(t)$ transverse to $\Gamma\cap E_1(t)$ at the point $p(s)\in \Gamma\cap E_1(t)$. Since $\Gamma$ intersects $\Sigma$ along $V_{\tau_1}'(t)$ and $V_{\tau_1}'(t)$ is compact, if $\epsilon$ is small, by standard arguments there exists a $C^{\infty}$ isotopy $\psi:\Sigma \times D_{\epsilon}\to M_{w}(t)$ with $\psi(\Sigma\times\{s\})=\Sigma_s$, $D_{\epsilon}=\{z\,|\, |z|<\epsilon\}$. In particular, the compactness of $V_{\tau_1}'(t)$ implies, via transversality theory, that $\Sigma_s$ intersects $\Gamma$ in a compact complex curve, say $\sigma_s$. Finally, $p(s)\in\Phi_t(\sigma_s)\subset Q$ and since 
$\Phi_t(\sigma_s)$ is a compact analytic subset of $Q$ we must have $\Phi_t(\sigma_s)=\{p(s)\}$, so that $\Phi_t^{-1}(p(s))=\sigma_s\subset \Gamma$, as asserted. 
\end{proof}

Now, the idea is to prove that there are two strips $\Gamma$ and $\Sigma$ as in Claim \ref{cl:38} 
such that  $\Gamma$ is a separatrix of $\mathcal F_t^{'}$ along $V_{\tau_1}'(t)$. 

By Claim \ref{cl:38}, $\Gamma$ will be also separatrix of  $\widetilde{\mathcal F}_t'$ and the two foliations will have a common leaf. This will conclude the proof of the Lemma \ref{l:36} and of Theorem \ref{teob}.  

In the construction of $\Gamma$ and $\Sigma$ as above, we will work with the deformation $t\in(\mathbb C,0)\to M_{w}(t)$. We can assume that there exists $\epsilon>0$ such that $(I), (II)$ and $(III)$ of Remark \ref{r:37} are true if $|t|<\epsilon$. Consider the complex orbifold $\widehat M_{w}(t)=\{(z,t)\,|\,|t|<\epsilon$ and $z\in M_{w}(t)$\}. Note that $\{(z,t)\in\widehat M_{w}(t)|t=t_0\}=M_{w}(t_0)\times\{t_0\}$, so that it will be denoted by $M_{w}(t_0)$. In $\widehat M_{w}$ consider the following objects:

\begin{itemize}
\item[(A).] The $2$-dimensional holomorphic foliation $\hat{\mathcal F^{'}}$ such that $\hat{\mathcal F_t^{'}}|_{M_{w}(t)}=\mathcal F_t^{'}$. Note that, by construction, the projection $(z,t)\in\widehat M_{w}\to t\in\mathbb C$ is a first integral of $\hat{\mathcal F^{'}}$.

\item[(B).] The $4$-dimensional subset $\hat E_1$ of $\widehat M_{w}$ defined by $\hat E_1\cap \widehat M_{w}(t)=E_1(t)$.
\item[(C).] The $1$-dimensional foliation $\hat{\mathcal G}$ on $\hat E_1$ defined by $\hat{\mathcal G} |_{E_1(t)}=\mathcal G_t$. Note that the projection $(z,t)\in\hat E_1\to t\in \mathbb C$ is a first integral of $\hat{\mathcal G}$.

\item[(D).] The map $\hat{\Phi}: \widehat M_{w}\to \hat E_1$ defined by $\hat{\Phi}(z,t)={\Phi}_t(z)\in E_1(t) \subset\hat E_1$.
\item[(E).] The $2$-dimensional  submanifold $\hat V_{\tau_1}'$ defined by $\hat V_{\tau_1}'\cap\widehat M_{w}(t)=V_{\tau_1}'(t)$.
\end{itemize}

The idea is to construct two germs of complex sumanifolds along $\hat V_{\tau_1}'$, say $\hat\Gamma$  and $\hat\Sigma$, such that:

\begin{itemize}
\item[(a).]  $\hat\Sigma$ is $\hat{\Phi}$-invariant.
\item[(b).] $\hat\Gamma$  and $\hat\Sigma$ are transverse and $\hat\Sigma\cap\hat\Gamma=\hat V_{\tau_1}'$.
\item[(c).] $\hat\Gamma\cap M_{w}(t)$ is a separatrix of $\mathcal F_t^{'}$ along $V_{\tau_1}'(t)$, if $|t|<\delta$, where $0<\delta\leq\epsilon$. In particular, $dim_{\mathbb C}(\hat\Gamma)=3$ and $dim_{\mathbb C}(\hat\Sigma)=4$.   
\end{itemize}

If $\hat\Gamma$ and $\hat\Sigma$ are as in (a), (b) and (c) then $\Gamma_t:=\hat\Gamma\cap M_{w}(t)$ and 
$\Sigma_t:=\hat\Sigma\cap M_{w}(t)$ satisfy the hypothesis of Claim \ref{cl:38}, if $|t|<\delta$, as the reader can check.
In the construction of $\hat\Sigma$ we will use that $\hat V_{\tau_1}'$ is contained in the Kupka set of $\hat{\mathcal F}^{'}$, which will be denoted by $K(\hat{\mathcal F}^{'})$. Indeed, $V_{\tau_1}'(t)\backslash\bigcup_{j} E_j(t)\subset K(\hat{\mathcal F}^{'})$ by Lemma \ref{subvariedades}, because ${\mathcal F}^{'}_t=\pi_{w}^\ast{\mathcal F_t}$. Moreover, for each $j\,=1,...,\frac{\nu^4}{\alpha\beta\gamma\delta}$, $V_{\tau_1}'(t)\cap E_j(t)$ consists of one point, which is also in the Kupka set because $p_{j}(t)$ is quasi-homogeneous singularity of $\mathcal F_t$. We leave the details to the reader. Given $w_0=(z_0,t_0)\in\hat V_{\tau_1}'$, if $\omega$ is a $2$-form  representing ${\mathcal F}^{'}_t$ in a neighborhood of $z_0$ then, by construction, the form $\hat\omega:=\omega\wedge dt$ represents $\hat{\mathcal F}^{'}$ in a neighborhood of $w_0$. Finally,  $d\hat\omega:=d\omega\wedge dt\neq0$, because $z_0$ is in the Kupka set of ${\mathcal F}^{'}_{t_0}$, which proves the assertion. 

Next we will find the normal type of $\hat{\mathcal F}^{'}$ at the point $(\tau_1(0),0)\in V_{\tau_1}'(0)\cap E_1(0)\subset \hat{E}_1$. Note that, by construction, this normal type coincides with the germ $\hat{\mathcal G}_0$ of $\mathcal G$ at $(\tau_1(0),0)$  (see (A), (B) and (C)). Since $(z,t)\to t$ is a first integral of $\hat{\mathcal G}$ and $\hat{\mathcal G}|_{E_{1}(t)}=\mathcal G_t$ , the normal type is done essentially by a holomorphic one-parameter family of vector fields $(Y_t)_{t\in(\mathbb C,0)}$, where $Y_t$ is a germ of vector field at $\tau_1(t)$ representing $\mathcal G_t$ at $\tau_1(t)$. For each fixed $t\in D_{\epsilon}$ the vector field $Y_t$ has $3$ one dimensional separatrices. In what follows from the theory of invariant manifolds of hyperbolic vector fields (see \cite{hps}) that it is possible to find a germ of holomorphic function $\lambda:(\mathbb C^2,0)\to\hat{E_1}$ such that 
\begin{itemize}
\item[(i).]   $\lambda(0,t)=\tau_1(t)$, $\forall t\in (\mathbb C,0)$.
\item[(ii).] $s\to\lambda(s,t)$, is a holomorphic parametrization of the separatrix $\lambda_1(t)$ of $\mathcal G_t$
\end{itemize}

Note that $\{\lambda(s,t)|s\in (\mathbb C,0), t=t_0\subset E_{1}(t_0)$, $\forall t_0$. Moreover, $\hat\lambda:=\lambda(\mathbb C^2,0)$ is a germ at $\tau_1(0)$ of a smooth surface in $\hat{E_1}$  and $\hat\lambda\cap E_1(t)$ is a separatrix of $\mathcal G_t$. for all $t\in (\mathbb C,0)$. Let us finish the construction of $\hat\Gamma$ using the property of local product along the Kupka set.  Since $V_{\tau_1}'$ is contained in the Kupka set of $\hat{\mathcal F}^{'}$, there exist a finite covering $(U_{\alpha})_{\alpha\in A}$ of $V_{\tau_1}'(0)$ $\hat{V}_{\tau_1}'$ by open sets of $\widehat M_{w}$ and a family of submersions $(\varphi_{\alpha})_{\alpha\in A}$, where $\varphi_{\alpha}:U_{\alpha}\to W$ ($W$ a neighborhood of $(\tau_1(0),0) \in \hat E_1)$ such that

\begin{equation}\label{eq:39}
\hat{\mathcal F}^{'}|_{U_{\alpha}}\,=\varphi_{\alpha}^{\ast}\,(\hat{\mathcal G}|_W)\,\,,\,\, \forall \alpha\in A.
\end{equation}

Fix $0<\delta\leq\epsilon$ such that $\lambda$ has representative defined in $D_{\delta}\times D_\delta$, denoted by the same symbol, such that $\hat\lambda_\delta:=\lambda(D_{\delta}\times D_\delta)\subset W$.

Let $\Gamma_\alpha$ be the germ of $\varphi_\alpha^{-1}(\hat\lambda_\delta)$ along $V_{\tau_1}'(0)\cap U_\alpha$. Note that for each $\alpha \in A$, $\Gamma_\alpha$ is a germ of smooth 3-dimensional manifold. 
Moreover, $\Gamma_\alpha$ is $\hat{\mathcal F}^{'}$-invariant, by \ref{eq:39}, so that $\Gamma_\alpha\cap M_{w}(t)$ is ${\mathcal F}^{'}(t)$-invariant, $\forall \in A$. Using again \ref{eq:39}, if $U_\alpha\cap U_\beta\neq\emptyset$ then $\Gamma_\alpha\cap U_\alpha\cap U_\beta=\Gamma_\beta\cap U_\alpha\cap U_\beta$. If we set $$\hat\Gamma=\bigcup_{\alpha}\Gamma_\alpha$$ then $\hat\Gamma$ satisfies (c): $\Gamma\cap M_{w}(t)$ is a separatrix of ${\mathcal F}^{'}(t)$ along $V_{\tau_1}'(t)$.

The construction of $\hat\Sigma$ as above is simpler. First of all, we fix any germ at $\tau_{1}(0)$ of smooth complex submanifold of $\hat E_1$, say $\hat C$, with the property that it is transverse to $\hat\lambda_\delta$ and $\hat C\cap\hat\lambda_\delta=\tau_1(0)$. Note that $dim_{\mathbb C}(\hat C)=3$. Set $\Sigma=\hat\Phi(\hat C)$. Now, use that for $t=0$ we have ${\mathcal F}^{'}_0=\tilde{\mathcal F}^{'}_0=\Phi_0^\ast(\mathcal G_0)$. This implies that $\hat\Sigma\cap M_w(0)$ is transverse to $\hat\Gamma\cap M_w(0)$ along $\hat{V}_{\tau_1}'(0)\cap M_w(0)$ along ${V}_{\tau_1}'(0)$. Therefore, by transversality theory, $\hat\Sigma\cap M_w(t)$ are transverse to and $\hat\Gamma\cap\hat\Sigma\cap M_w(t)={V}_{\tau_1}'(t)$, as wished. 

This finishes proof of Theorem \ref{teob} in the case (1).

The other cases are analogous.

\section{Remarks and complements}

In \cite{soares1}  the author shows  that  for a fixed degree the union of the
logarithmic components of the space of codimension one foliations of a fixed degree $\mathbb P^{n}$; $n\geq3$ is connected. 

Theorem \ref{teob} motivates the following:

\begin{prob}\label{pr:1}
{\rm Is it true that for a fixed degree, the union of the branched pull-back components is a connected subset?}
\end{prob}

\bibliographystyle{amsalpha}

\end{document}